\documentclass[12pt]{article}
\usepackage{epsfig}
\usepackage{fancyhdr}

\usepackage{amssymb}
\usepackage{amsmath}

\usepackage{graphicx}
\usepackage[usenames,dvipsnames]{color}
\newcommand\C{{\cal C}}

\newcommand\D{{\cal D}}

\newcommand\N{{\cal N}}

\newtheorem{theorem}{Theorem}[section]
\newtheorem{lemma}[theorem]{Lemma}
\newtheorem{corollary}[theorem]{Corollary}

\newtheorem{definition}[theorem]{Definition}

\newtheorem{construction}{Construction}

\newenvironment{proof}{\noindent{\bf Proof}\hspace{0.5em}}
    { \null  \hfill $\square$ \par}

\usepackage[mathscr]{eucal}%Euler script, use \mathscr

\renewcommand{\gg}{g}

\newcommand{\R}{\mathcal R}

\renewcommand{\S}{\mathcal S}
\renewcommand{\P}{\mathcal P}

\newcommand{\abb}{{\cal A(\cal S)}}
\newcommand{\pbb}{{\cal P(\cal S)}}
\newcommand{\Bpi}{\mathscr B}

\newcommand{\ES}{{\mathbb S}}
\newcommand{\ET}{\mathbb T}
\newcommand{\EC}{\mathbb C}

\newcommand{\EX}{\mathbb X}
\newcommand{\EY}{\mathbb Y}

\newcommand{\gt}{g_\ET}
\newcommand{\gc}{g_\EC}
\newcommand{\gs}{g_\ES}

\newcommand{\TP}{{\mathcal T}_P}

\renewcommand{\r}{{q}}

\newcommand\w{\tau}

\newcommand{\si}{\Sigma_\infty}
\newcommand{\li}{\ell_\infty}

\newcommand{\orsps}{order-$\r$-subplanes}
\newcommand{\orsls}{order-$\r$-sublines}
\newcommand{\orsp}{order-$\r$-subplane}
\newcommand{\orsl}{order-$\r$-subline}
\newcommand{\psline}{pencil-subline}
\newcommand{\dcsline}{dual-conic-subline}

\newcommand{\elltau}{\ell= [-\tau\tau^q, \tau^q+\tau,-1]}

\newcommand{\takeaway}{\backslash}
\newcommand{\st}{:}

\newcommand\PGL{{\rm PGL}}
\newcommand\GF{{\rm GF}}
\newcommand\PG{{\rm PG}}

\newcommand{\car}{E} %carrier points

\renewcommand\setminus{\backslash}
\newcommand{\Label}{\label}

\usepackage{setspace}

\begin{document}
%Fancy header usage
%\pagestyle{fancy}
%\fancyhf{}
%\fancyhead[R]{} % predefined ()
%\fancyhead[L]{\color{Magenta}\leftmark} % 1. sectionname
%\fancyfoot[C]{\color{Magenta}\thepage}
%\lfoot{\color{Magenta}\jobname}
%\rfoot{\color{Magenta}\today}
%
%\fancypagestyle{plain}{%
%  \fancyhf{}%
%  \renewcommand{\headrulewidth}{0pt}%
%}

%{\LARGE{\color{Magenta}
%EXT3: The exterior splash in $\PG(6,q)$: Special conics}}
%
%
% \begin{spacing}{0} \tableofcontents \end{spacing}
%%\newpage
%

\title{
The exterior splash in $\PG(6,q)$: Special conics}

\author{S.G. Barwick and Wen-Ai Jackson
%\date{\today}
%\\ Department of Pure Mathematics, University of Adelaide\\
%Adelaide 5005, Australia
%\\ \\
}

\maketitle

%Corresponding Author: Dr Susan Barwick, University of Adelaide, Adelaide
%5005, Australia. Phone: +61 8 8303 3983, Fax: +61 8 8303 3696, email:
%susan.barwick@adelaide.edu.au

%Keywords: Bruck-Bose representation, subplanes, exterior splash
%AMS code: 51E20

\begin{abstract}
Let $\pi$ be an order-$q$-subplane of $\PG(2,q^3)$ that is exterior to $\ell_\infty$. The exterior splash of $\pi$ is the set of $q^2+q+1$ points on $\ell_\infty$ that lie on an extended line of $\pi$. Exterior splashes are projectively equivalent to scattered linear sets of rank 3, covers of the circle geometry $CG(3,q)$, and hyper-reguli of $\PG(5,q)$. In this article we use the Bruck-Bose representation in $\PG(6,q)$ to give a geometric characterisation of special conics of $\pi$ in terms of the covers of the exterior splash of $\pi$. We also investigate properties of \orsps\ with a common exterior splash, and study the intersection of two exterior splashes. 
\end{abstract}

\section{Introduction}

Let $\pi$ be 
a subplane of $\PG(2,q^3)$ of order $q$ that is exterior to $\li$. We call $\pi$ an {\em exterior \orsp\ of $\PG(2,q^3)$}. The lines of $\pi$ meet $\li$ in a set $\ES$ of size $q^2+q+1$, called the {\em exterior splash} of $\pi$ on $\li$. 
The exterior splash turns out to be a set rich in geometric structure. In particular, \cite{BJ-ext1} showed that the sets of points in an exterior splash has arisen in many different situations, namely scattered linear sets of rank 3, Sherk surfaces of size $q^2+q+1$, covers of the circle geometry $CG(3,q)$, and so hyper-reguli in $\PG(5,q)$.  
Properties of the exterior splash in the Bruck-Bose representation in $\PG(6,q)$ are studied in \cite{BJ-ext2}. 
This current article continues this study, and furthers the investigation into the interplay between an exterior \orsp\  in $\PG(2,q^3)$, and its associated exterior splash.

This article proceeds as follows. 
In Section~\ref{sect:intro} we
introduce the notation we use for 
 the Bruck-Bose representation of $\PG(2,q^3)$ in $\PG(6,q)$, 
and give the  background on exterior splashes needed to understand the results of this article. Section 3 looks at the Bruck-Bose representation of a class of special conics in an exterior \orsp\ of $\PG(2,q^3)$.  In particular, Section 3.2 contains the main characterisation of this article. We show that the $(\pi,\li)$-special conics in an exterior \orsp\ $\pi$ of $\PG(2,q^3)$ correspond exactly to the $\EX$-special twisted cubics in $\PG(6,q)$, where $\EX$ is a cover of an exterior splash.
We then explore some of the fundamental geometric differences
between the two covers of an exterior splash with respect to an associated \orsp.
Section 4 investigates \orsps\ with a common splash, and Section 5 looks at the intersection of two exterior splashes. 

\section{Background results}\Label{sect:intro}

\subsection{The Bruck-Bose representation of $\PG(2,q^3)$ in $\PG(6,q)$}\Label{BBintro}

Here we  introduce the
notation we will use for the 
Bruck-Bose representation of $\PG(2,q^3)$ in $\PG(6,q)$. We work with the finite field ${\mathbb F}_{q}=\GF(q)$, and let ${\mathbb F}_{q}'={\mathbb F}_{q}\setminus\{0\}$.
A 2-{\em spread} of $\PG(5,\r)$ is a set of $\r^3+1$ planes that partition
$\PG(5,\r)$. 
A 2-{\em regulus} of $\PG(5,\r)$ is a
set of $\r+1$ mutually disjoint planes $\pi_1,\ldots,\pi_{\r+1}$ with
the property that if a line meets three of the planes, then it meets all
${\r+1}$ of them. A 2-regulus $\R$ has a set of $q^2+q+1$ mutually disjoint {\em ruling lines} that meet every plane of $\R$.  A $2$-spread $\S$ is {\em regular} if for any three planes in $\S$, the
$2$-regulus containing them is contained in $\S$. 
See \cite{hirs91} for
more information on $2$-spreads.

The following construction of a regular $2$-spread of $\PG(5,\r)$ will be
needed. Embed $\PG(5,\r)$ in $\PG(5,\r^3)$ and let $g$ be a line of
$\PG(5,\r^3)$ disjoint from $\PG(5,\r)$. Let $g^\r$, $g^{\r^2}$ be the
conjugate lines of $g$; both of these are disjoint from $\PG(5,\r)$. Let $P_i$ be
a point on $g$; then the plane $\langle P_i,P_i^\r,P_i^{\r^2}\rangle$ meets
$\PG(5,\r)$ in a plane. As $P_i$ ranges over all the points of  $g$, we get
$\r^3+1$ planes of $\PG(5,\r)$ that partition $\PG(5,\r)$. These planes form a
regular $2$-spread $\S$ of $\PG(5,\r)$. The lines $g$, $g^\r$, $g^{\r^2}$ are called the (conjugate
skew) {\em transversal lines} of the $2$-spread $\S$. Conversely, given a regular $2$-spread
in $\PG(5,\r)$,
there is a unique set of three (conjugate skew) transversal lines in $\PG(5,\r^3)$ that generate
$\S$ in this way.

We will use the linear representation of a finite
translation plane $\P$ of dimension at most three over its kernel,
due independently to
Andr\'{e}~\cite{andr54} and Bruck and Bose
\cite{bruc64,bruc66}. 
Let $\si$ be a hyperplane of $\PG(6,\r)$ and let $\S$ be a $2$-spread
of $\si$. We use the phrase {\em a subspace of $\PG(6,\r)\takeaway\si$} to
  mean a subspace of $\PG(6,\r)$ that is not contained in $\si$.  Consider the following incidence
structure:
the \emph{points} of $\abb$ are the points of $\PG(6,\r)\takeaway\si$; the \emph{lines} of $\abb$ are the $3$-spaces of $\PG(6,\r)\takeaway\si$ that contain
  an element of $\S$; and \emph{incidence} in $\abb$ is induced by incidence in
  $\PG(6,\r)$.
Then the incidence structure $\abb$ is an affine plane of order $\r^3$. We
can complete $\abb$ to a projective plane $\pbb$; the points on the line at
infinity $\li$ have a natural correspondence to the elements of the $2$-spread $\S$.
The projective plane $\pbb$ is the Desarguesian plane $\PG(2,\r^3)$ if and
only if $\S$ is a regular $2$-spread of $\si\cong\PG(5,\r)$ (see \cite{bruc69}).
For the remainder of the article, we use $\S$ to denote a regular 2-spread of $\si\cong\PG(5,q)$.

We use the following notation.  If $T$
is a point of $\li$ in $\PG(2,q^3)$, we use $[T]$ to refer to the plane of $\S$ corresponding to $T$. More generally, if $X$
is a set of points of $\PG(2,q^3)$, then we let $[X]$ denote the
corresponding set  in $\PG(6,q)$. If $P$ is an affine point of $\PG(2,q^3)$,  we generally simplify the notation and also
use $P$  to refer to the corresponding affine point in $\PG(6,q)$, although in  some cases, to avoid confusion, we use $[P]$.

When $\S$ is a regular 2-spread, 
we can relate the coordinates of $\pbb\cong\PG(2,\r^3)$ and $\PG(6,\r)$ as
follows. Let $\tau$ be a primitive element in ${\mathbb F}_{q^3}$ with primitive
polynomial $x^3-t_2x^2-t_1x-t_0.$ Then every element $\alpha\in{\mathbb F}_{q^3}$
can be uniquely written as $\alpha=a_0+a_1\tau+a_2\tau^2$ with
$a_0,a_1,a_2\in{\mathbb F}_{q}$. Points in $\PG(2,\r^3)$ have homogeneous coordinates
$(x,y,z)$ with $x,y,z\in{\mathbb F}_{q^3}$. Let the line at infinity $\li$ have
equation $z=0$; so the affine points of $\PG(2,\r^3)$ have coordinates
$(x,y,1)$. Points in $\PG(6,\r)$ have homogeneous coordinates
$(x_0,x_1,x_2,y_0,y_1,y_2,z)$ with $x_0,x_1,x_2,y_0,y_1,y_2,z\in{\mathbb F}_{q}$.  Let $\si$ have equation $z=0$. 
Let $P=(\alpha,\beta,1)$ be a point of $\PG(2,\r^3)$. We can write $\alpha=a_0+a_1\tau+a_2\tau^2$ and
$\beta=b_0+b_1\tau+b_2\tau^2$ with $a_0,a_1,a_2,b_0,b_1,b_2\in{\mathbb F}_{q}$. 
We want to map the element $\alpha$ of ${\mathbb F}_{q^3}$ to the vector $(a_0,a_1,a_2)$, and we use the following notation to do this:
$$[\alpha]=(a_0,a_1,a_2).$$
This gives us some notation for the Bruck-Bose map, denoted $\epsilon$, from an affine point $P=(\alpha,\beta,1)\in\PG(2,\r^3)\takeaway\li$ to the corresponding affine point $[P]\in\PG(6,\r)\takeaway\Sigma_\infty$, namely  $\epsilon(\alpha,\beta,1)=[(\alpha,\beta,1)]=([\alpha],[\beta],1)=(a_0,a_1,a_2,b_0,b_1,b_2,1).$
More generally, if $z\in{\mathbb F}_{q}$, then 
$\epsilon(\alpha,\beta,z)=([\alpha],[\beta],z)=(a_0,a_1,a_2,b_0,b_1,b_2,z).$

Consider the case when $z=0$, a point on $\li$ in $\PG(2,q^3)$ has coordinates
 $L=(\alpha,\beta,0)$ for some $\alpha,\beta\in{\mathbb F}_{q^3}$.
 In $\PG(6,q)$, the point 
 $\epsilon(\alpha,\beta,0)=([\alpha],[\beta],0)$  is one point in the spread element $[L]$
corresponding to $L$. Moreover, the spread element $[L]$ consists of all the points $\{([\alpha x],[\beta x],0)\st x\in{\mathbb F}_{q^3}'\}$.

\subsection{Introducing exterior splashes}\Label{sec:intro-ext}

This article relies heavily on properties of exterior splashes proved in \cite{BJ-ext1,BJ-ext2}. In this section we detail the most important definitions and results that are needed in this article. We begin with the important result that all exterior splashes are projectively equivalent. 

\begin{theorem}\Label{transsplashes}{\rm\cite{BJ-ext1}}
Consider the collineation group  $G=\PGL(3,\r^3)$ acting on $\PG(2,q^3)$. The 
subgroup $G_\ell$ fixing a line $\ell$ is transitive on the \orsps\ that are exterior to $\ell$, and is transitive on the
exterior splashes of $\ell$.
\end{theorem}

This theorem says that if we want to prove a result about exterior \orsps, or about exterior splashes, then we can without loss of generality prove it for a particular \orsp. In \cite{BJ-ext2},  an \orsp\ $\Bpi$ which is exterior to $\li$ is coordinatised, and $\Bpi$ is used in many of the proofs in this article. The main points of this coordinatisation are as follows. 

\begin{theorem}\Label{thm:Bpi}{\rm\cite{BJ-ext2}}
 Let $\sigma$ be the homography of $\PG(2,q^3)$ with matrix
\begin{eqnarray*}
K=\begin{pmatrix}-\tau&1&0\\ -\tau^q&1&0\\ \tau\tau^q&-\tau-\tau^q&1\end{pmatrix}. 
\end{eqnarray*}
Then $\sigma$ maps the \orsp\ $\pi_0=\PG(2,q)$ with exterior line $\elltau$ to an \orsp\ $\Bpi=\sigma(\pi_0)$ with exterior line  $\li=\sigma(\ell)=[0,0,1]$.
Further, $\Bpi$ has exterior splash 
$\ES=\{ (k,1,0)\st k\in{\mathbb F}_{q^3}, 
k^{q^2+q+1}=1\}\equiv\{(\tau^{(q-1)i},1,0)\st 0\le i < q^2+q+1\}$
 and carriers $\car_1=(1,0,0)$ and $\car_2=(0,1,0)$.
 \end{theorem}

We will also need the following result about exterior \orsps\ which have  a common exterior splash and a common \orsl. 

\begin{theorem}\Label{two-subplanes}{\rm\cite{BJ-ext1}}
In $\PG(2,q^3)$, let $\ES$ be an exterior splash of $\li$, let $\ell$ be a line through a point of $\ES$, and let  $b$ an \orsl\ of $\ell$ exterior to $\li$. Then there are exactly two \orsps\ that contain $b$ and have exterior splash $\ES$.
\end{theorem}

We now consider an exterior splash $\ES$ of $\li$ in $\PG(2,q^3)$, and look at it in the Bruck-Bose representation in $\PG(6,q)$. The set of $q^2+q+1$  points $\ES\subset\li$ corresponds to a set of $q^2+q+1$ planes (which we also denote by $\ES$) of the regular 2-spread $\S$ in $\si\cong\PG(5,q)$. As an exterior splash is equivalent to a cover of the circle geometry $CG(3,q)$, by \cite{bruc73b},  $\ES$ has two {\em switching sets} denoted $\EX,\EY$, each containing $q^2+q+1$ planes of $\si$. The three sets $\ES,\EX,\EY$ are called {\em hyper-reguli} in \cite{ostrom}: each set  covers 
the same set of points;  two planes in the same set are disjoint; while two planes from different sets meet in a unique point. Note that in \cite{BJ-short}, we show that the only planes which meet every plane of $\ES$ in a point are those in $\EX$ and $\EY$, so in particular, $\ES$ has exactly two covers.   In this article we call $\EX$ and $\EY$ {\em covers} of the exterior splash $\ES$. If $\pi$ is an exterior \orsp\ of $\PG(2,q^3)$ with exterior splash $\ES$, then the two covers have different properties with respect to $\pi$. One cover is denoted $\ET$ and is called the {\em tangent cover} of $\pi$ (or the tangent cover of $\ES$ with respect to $\pi$). The name follows from \cite[Theorem 5.3]{BJ-ext2} which shows that the planes of $\ET$ are related to the tangent planes of $[\pi]$. The other cover is denoted $\EC$ and is called the {\em conic cover} of $\pi$, it is related to certain conics of $\pi$, see Theorem~\ref{conictonrc}. In  Section~\ref{sec:sconic-stc} we discuss some of the fundamental differences in the geometry of the two covers in relation to the \orsp\ $\pi$.

In \cite{BJ-ext2}, we show that in the cubic extension $\PG(5,q^3)$ of $\si\cong\PG(5,q)$, each set $\ES,\ET,\EC$ has a unique triple of conjugate {\em transversal lines}. That is, the lines $\gs,\gs^q,\gs^{q^2}$ of $\PG(5,q^3)\setminus\PG(5,q)$ meet every extended plane of $\ES$, and these are the only lines of $\PG(5,q^3)$ which meet every extended plane of $\ES$ (in fact, these are the three transversal lines of the regular 2-spread $\S$). Similarly, there are exactly three lines 
 $\gt,\gt^q,\gt^{q^2}$ of $\PG(5,q^3)\setminus\PG(5,q)$ that meet every extended plane of the cover $\ET$; and exactly three lines   $\gc,\gc^q,\gc^{q^2}$ of $\PG(5,q^3)\setminus\PG(5,q)$ that meet every extended plane of the cover $\EC$.
  
We define the notion of $\EX$-special conics/twisted cubics in $\PG(6,q)$. In this article, this definition is used when $\EX$ is the regular 2-spread $\S$, or an exterior splash $\ES$, or one of the covers $\ET$ or $\EC$ of an exterior splash.

\begin{definition}\Label{def:S-special}
\begin{enumerate}
\item 
An \,{\em
   $\EX$-special conic}  is a non-degenerate conic $\C$ contained in a plane of \,$\EX$, such
that 
the extension of $\C$ to $\PG(6,q^3)$
meets the three transversal lines of 
 $\EX$. 
 \item An \,{\em $\EX$-special twisted cubic} is a
twisted cubic  $\N$ in a $3$-space of $\PG(6,q)\backslash\si$ about a plane of \,$\EX$, such that the extension of $\N$ to $\PG(6,q^3)$ meets the three transversal lines of $\EX$. 
\end{enumerate}
\end{definition}

Note that an $\EX$-special twisted cubic has no points in $\si$. 
Finally, we need the following result from \cite{barw12} which describes the representation of certain  \orsls\ of $\PG(2,q^3)$ in the Bruck-Bose representation in $\PG(6,q)$.

\begin{theorem}\Label{sublinesinBB}
\begin{enumerate}
\item Order-$q$-sublines contained in $\li$ in $\PG(2,q^3)$ correspond exactly to $2$-reguli of $\S$ in $\PG(6,q)$. 
\item Order-$q$-sublines exterior to $\li$ in $\PG(2,q^3)$ correspond exactly to
 $\S$-special twisted cubics in  $\PG(6,q)$.
\end{enumerate}
\end{theorem}

\section{Special conics of an exterior subplane}\Label{sec:sconic-stc}

The main result of this section is to show that a certain  type of conic in an exterior \orsp\ in $\PG(2,q^3)$ corresponds exactly to $\EC$-special twisted cubics in $\PG(6,q)$. We first introduce the idea  in $\PG(2,q^3)$ of a special conic in an exterior \orsp.

\subsection{Special conics}

We are interested in a particular class of  conics in an exterior \orsp\ in $\PG(2,q^3)$ which have a nice geometric representation in $\PG(6,q)$. We define these special conics as follows. Note first that we can generalise the notion of an exterior splash on $\li$ to define the exterior splash $\ES$ of an \orsp\ $\pi$ onto any exterior line $\ell$:  $\ES$ is  the set of $q^2+q+1$ points on $\ell$ that lie on an extended line of $\pi$. 
An important collineation group  acting on $\PG(2,q^3)$ is $I=\PGL(3,q^3)_{\pi,\ell}$ which 
fixes an \orsp\ $\pi$, and a line $\ell$ exterior to $\pi$. 
By \cite[Theorem 2.2]{BJ-ext1}, $I$ fixes exactly three lines, namely $\ell$ and its conjugates $m$, $n$ with respect to $\pi$. Further   
$I$  fixes exactly three points: $\car_1=\ell\cap m$, 
$\car_2=\ell\cap n$, 
$\car_3=m \cap n$, which are conjugate with respect to $\pi$.
Note that if $\pi_0=\PG(2,q)$, then  $I=\PGL(3,q^3)_{{\pi_0},\ell}$ fixes the three exterior lines $\ell,\ell^q,\ell^{q^2}$ and the three points $E=\ell\cap\ell^{q^2}$, $E^q$, $E^{q^2}$.

The group $I=\PGL(3,q^3)_{\pi,\ell}$ identifies  two special points $\car_1=\ell\cap m$, 
$\car_2=\ell\cap n$ on $\ell$ which are called the {\em carriers} of  the exterior splash $\ES$ of $\pi$ onto $\ell$. This is consistent with the definition of carriers of a circle geometry $CG(3,q)$, see \cite[Theorem 4.2]{BJ-ext1}. The fixed points and fixed lines of $I$ are used to define an important class of conics in  $\pi$.

 \begin{definition}
 A $(\pi,\ell)$-special conic of an \orsp\ $\pi$ (with exterior line $\ell$) is a conic of $\pi$ whose extension to $\PG(2,q^3)$ contains the three fixed points $E_1,E_2,E_3$ of $I=\PGL(3,q^3)_{\pi,\ell}$.
\end{definition}

It is straightforward to show that a $(\pi,\ell)$-special conic is irreducible. 
We note that the incidence structure with {\sl points} the points of an exterior \orsp\  $\pi$, {\sl lines} the $(\pi,\li)$-special conics of $\pi$, and natural incidence is isomorphic to $\PG(2,q)$. 
This observation is useful when counting special conics in this article.  We also note that the set of $q^2+q+1$ $(\pi,\li)$-special conics of $\pi$ form a circumscribed bundle of conics in the sense of \cite{BBEF}.

In the case when $q$ is even, the set of special conics in an exterior \orsp\  
 has some interesting properties relating to the nuclei.
We look at the structure of the set of $(\pi,\li)$-special conics through a fixed point $P$ in $\pi$.

\begin{theorem}  \Label{qeven-spec-con}
Let $\pi$ be an exterior \orsp\ of $\PG(2,q^3)$, $q$ even. 
Let $P$ be a point of $\pi$, let $\C_0,\ldots,\C_q$ be the $(\pi,\li)$-special conics of $\pi$ through $P$, and let $N_k$ be the nucleus of $\C_k$. Then the points $N_0,\ldots,N_q$ are distinct, and lie on a line $n_P$ not through $P$, called the {\sl nucleus line} of $P$. Further, every line of $\pi$ is the nucleus line for a distinct point  of $\pi$. 
\end{theorem}

\begin{proof}
By Theorem~\ref{transsplashes} and \cite[Theorem 2.2]{BJ-ext1}, we can without loss of generality prove this for the \orsp\ $\pi_0=\PG(2,q)$ exterior to the line $\elltau$, and the point $P=(0,0,1)\in\pi_0$. The fixed lines of $I=\PGL(3,q^3)_{\pi_0,\ell}$ are $\ell,\ell^q,\ell^{q^2}$, and the fixed points of $I$ are $\car=\ell\cap\ell^{q^2}=(1,\tau,\tau^2)$, $\car^q=\ell^q\cap\ell$ and $E^{q^2}=\ell^{q^2}\cap\ell$. 
The conics $\C_0\colon y^2-xz=0$ and $\C_\infty \colon -t_0x^2+yz-t_2xz-t_1xy=0$ are 
conics of $\pi_0$ which contain the four points $P$, $\car$, $\car^q$ and $\car^{q^2}$\!\!, so they are 
 $(\pi_0,\ell)$-special conics containing $P$.  
Hence the conics in the pencil $\{\C_k=\C_0+k\C_\infty\st k\in{\mathbb F}_{q}\cup\{\infty\}\}$ are the $q+1$  $(\pi_0,\ell)$-special conics of $\pi_0$ containing $P$.
The nucleus of conic $\C_k$ is $N_k=(k,1-kt_2,-kt_1)$, $k\in{\mathbb F}_{q}\cup\{\infty\}$. So the nuclei are distinct for distinct $k$, and all lie on the line $n_P=[t_1,0,1]$, a line not through $P$. 

Let $P,Q$ be two points of $\pi_0$ with nucleus lines $n_P,n_Q$ respectively. We show that $n_P,n_Q$ are distinct.  Let $N\in n_P\cap n_Q$. There is a unique $(\pi_0,\ell)$-special conic $\C$ with nucleus $N$. As $N\in n_P$, $\C$ contains $P$, and as  $N\in n_Q$, $\C$ contains $Q$. So $\C$ is the unique $(\pi_0,\ell)$-special conic through $P$ and $Q$. Hence there can only be one point in the intersection $n_P\cap n_Q$, so the lines $n_P,n_Q$ are distinct. Hence each line of $\pi_0$ is the nucleus line for a distinct point of $\pi_0$. \end{proof}

A consequence of this theorem is that there is a bijection between 
 lines of an \orsp\ $\pi$ and $(\pi,\ell)$-special conics of $\pi$ as follows.
Firstly, there is a bijection mapping a point $P\in\pi$ to the unique $(\pi,\ell)$-special conic of $\pi$ with nucleus $P$. Then Theorem~\ref{qeven-spec-con} gives a bijection mapping a point $P\in\pi$ to its nucleus line $n_P$.

\subsection{Special conics are special twisted cubics}\Label{sec:sp-con}

Let $\ES$ be an exterior splash of $\PG(6,q)$ with covers $\EC$ and $\ET$, so  $\ES$ is  contained in the regular $2$-spread $\S$.
In Definition~\ref{def:S-special} we defined the notion of $\S$-, $\ES$-, $\EC$- and $\ET$-{\em special twisted cubics}. We consider the representation of each of these special twisted cubics in $\PG(2,q^3)$.

We can characterise  $\S$-special and $\ES$-special twisted cubics as follows. By Theorem~\ref{sublinesinBB}, a twisted cubic $\N$ of $\PG(6,q)$ corresponds to an exterior \orsl\ of $\PG(2,q^3)$ if and only if $\N$ is an $\S$-special twisted cubic. Hence if $\pi$ is an exterior \orsp\ of $\PG(2,q^3)$ with exterior splash $\ES$, then  the \orsls\ of $\pi$ are precisely the $\ES$-special twisted cubics in $\PG(6,q)$. 

In this section, we characterise $\EC$- and $\ET$-special twisted cubics, and show they correspond  precisely to 
  special conics of  some exterior \orsp\ of $\PG(2,q^3)$. 
 The argument  proving this proceeds as follows. In Theorem~\ref{conictonrc}, we show that a $(\pi,\li)$-special conic $\C$ in an exterior \orsp\ $\pi$ of $\PG(2,q^3)$ corresponds to a twisted cubic $[\C]$ in $\PG(6,q)$. Moreover, the 3-space containing $[\C]$ contains a unique plane of the cover $\EC$ (this is why we named $\EC$ the `conic cover').
 Theorem~\ref{two-specials} shows that the twisted cubic $[\C]$  is $\EC$-special. Theorem~\ref{two-specials-conv} proves the converse, that each $\EC$-special twisted cubic of $\PG(6,q)$  corresponds to a $(\pi',\li)$-special conic in some exterior \orsp\ $\pi'$ of $\PG(2,q^3)$. Note also that Corollary~\ref{6dim-nrc} shows that a non-special conic of $\pi$ corresponds to a 6-dimensional normal rational curve of $\PG(6,q)$. Finally, Corollary~\ref{cor:t-special-tc} shows that every
$\ET$-special twisted cubic  meets $\pi$ in at most three points, but corresponds to a special conic in some other exterior \orsp\ of $\PG(2,q^3)$. 

\begin{theorem}\Label{conictonrc}
Let $\C$ be a  $(\pi,\li)$-special conic in an exterior \orsp\ $\pi$ of $\PG(2,\r^3)$. Then in $\PG(6,q)$, $\C$ corresponds to a twisted cubic $[\C]$, and the $3$-space containing $[\C]$ meets $\si$ in a plane of the conic cover $\EC$ of $\pi$. 
\end{theorem}

\begin{proof}
By Theorem~\ref{transsplashes}, we can without loss of generality prove this for the \orsp\ $\Bpi$ coordinatised in Theorem~\ref{thm:Bpi}. Further, by \cite[Theorem 6.1]{BJ-ext1}, we can
without loss of generality prove it for any $(\Bpi,\li)$-special conic in $\Bpi$.
In $\PG(2,q^3)$, consider the \orsp\ $\pi_0=\PG(2,q)$ with exterior line $\elltau$, and let $\C'$ be the conic in $\pi_0$ of equation $y^2=zx$. By \cite[Lemma 2.4]{BJ-ext1}, the exterior splash of $\pi_0$ onto $\ell$ has carriers $E=(1,\tau,\tau^2)$ and $E^q$. These both lie on the extension of $\C'$ to $\PG(2,q^3)$, as does $E^{q^2}$, hence  $\C'$ is a $(\pi_0,\ell)$-special conic of $\pi_0$.
By Theorem~\ref{thm:Bpi},  the homography $\sigma$ with matrix $K$ maps $\pi_0$, $\ell$  to $\Bpi$, $\li$ respectively. Further,  $\C=\sigma(\C')$ is a $(\Bpi,\li)$-special conic in $\Bpi$. 
The points of $\C'$ are $(1,\theta,\theta^2)$, $\theta\in{\mathbb F}_{q}\cup\{\infty\}$, so $\C$ has points $P_\theta=K(1,\theta,\theta^2)^t$, $\theta\in{\mathbb F}_{q}\cup\{\infty\}$. Now $P_\infty=(0, 0,  1)$, and for $\theta\in{\mathbb F}_{q}$, we have
\begin{eqnarray*}
P_\theta&=&\big(-\tau+\theta,\ -\tau^q+\theta,\  \tau\tau^q-\theta(\tau^q+\tau)+\theta^2\big)\\
&\equiv&\big(-(\theta-\tau)(\theta-\tau^{q^2}),\ -(\theta-\tau^q)(\theta-\tau^{q^2}),\  (\theta-\tau)(\theta-\tau^q)(\theta-\tau^{q^2})\big).\end{eqnarray*}
Note that the first two coordinates are polynomials in $\theta$ of degree 2, and the third coordinate is a polynomial in $\theta$ of degree 3. We
 write $P_\theta=(G(\theta)^q,G(\theta),f(\theta))$ where 
 $G(\theta)=-(\theta-\tau^q)(\theta-\tau^{q^2})$ is a polynomial over ${\mathbb F}_{q^3}$ and $f(\theta)=(\theta-\tau)^{q^2+q+1}$ is a polynomial over ${\mathbb F}_{q}$.
For $\theta\in{\mathbb F}_q$,  the third coordinate $f(\theta)$ is an element of ${\mathbb F}_{q}$, hence the corresponding point in $\PG(6,q)$ is  $P_\theta=([G(\theta)^{q}],[G(\theta)],f(\theta))$. 
As each coordinate of  $P_\theta$ is  a polynomial in $\theta$ of degree at most three, and the coordinates have no common factors over ${\mathbb F}_{q}$, the set of points $[\C]=\{P_\theta$: $\theta\in{\mathbb F}_{q}\}\cup\{P_\infty=(0, 0,  1)\}$  is a  rational curve of order three in $\PG(6,q)$.

We show that $[\C]$ is normal in a 3-space  by 
 considering the 
projection of $[\C]$ onto $\si
$ from the point $P_\infty$. 
Straightforward calculations show that the  point $Q_\theta=P_\infty P_\theta\cap\Sigma_\infty$ has coordinates $Q_\theta=([G(\theta)^q],\ [(G(\theta)^q)^{q^2}],\ 0)$.
By \cite[Lemma 5.1]{BJ-ext2}, one of  the planes  of the conic cover $\EC$ of $\Bpi$ is 
$[C_1]=\{([x],[x^{q^2}],0)\st x\in{\mathbb F}_{q^3}'\}$. So the point $Q_\theta$ lies in the cover plane $[C_1]$ for all $\theta\in{\mathbb F}_q$.  
We use the identities 
$\tau^q+\tau^{q^2}=t_2-\tau$ and $\tau^q\tau^{q^2}=\tau^2-t_2\tau-t_1$ to  simplify 
$G(\theta$) to $G(\theta)=(-\theta^2+t_2\theta+t_1)+(-\theta+t_2)\tau-\tau^2$.  
So $[G(\theta)]=\theta^2(-1,0,0)+\theta(t_2,-1,0)+(t_1,t_2,-1)$. Hence the points $Q_\theta$, $\theta\in{\mathbb F}_q$, lie on an irreducible conic in $[C_1]$. 
Hence the points of $[\C]$ span the $3$-space $\langle P_\infty,[C_1]\rangle$, and so  the conic $\C$ corresponds to a twisted cubic $[\C]$ in a $3$-space that  meets $\Sigma_\infty$ in a plane of $\EC$. 
\end{proof}

\begin{corollary}\Label{s-con-in-distinct}
Let $\pi$ be an exterior \orsp\ of $\PG(2,q^3)$ with conic cover $\EC$.  Distinct  $(\pi,\li)$-special conics in $\pi$ correspond to  twisted cubics lying in  $3$-spaces through  distinct planes of $\EC$.
\end{corollary}

\begin{proof} By \cite[Theorem 6.1]{BJ-ext1}, 
 the singer cycle $I=\PGL(3,q^3)_{\Bpi,\li}$ 
acts regularly on the $(\Bpi,\li)$-special conics of $\Bpi$. Further, by \cite[Lemma 5.2]{BJ-ext2}, in $\PG(6,q)$, $[I]$ is a singer cycle  that  acts regularly on the cover planes of $\EC$.
There are $q^2+q+1$ planes in $\EC$, and by \cite{BJ-ext1}, there are $q^2+q+1$ $(\pi,\li)$-special conics in  $\pi$.
Hence every $(\pi,\li)$-special conic corresponds to a twisted cubic that lies in a $3$-space about a distinct plane of the conic cover $\EC$. 
\end{proof}

We now show that the twisted cubics of $\PG(6,q)$ corresponding to $(\pi,\ell)$-special conics of $\pi$ in $\PG(2,q^3)$ are $\EC$-special twisted cubics. That is, in the cubic extension $\PG(6,q^3)$, they meet the transversals $g_\EC,g_\EC^q,g_\EC^{q^2}$ of the conic cover $\EC$ of $\pi$.

\begin{theorem}\Label{two-specials}
Let $\pi$ be an exterior \orsp\ of $\PG(2,q^3)$ with conic cover $\EC$.
A $(\pi,\li)$-special conic of $\pi$ corresponds in $\PG(6,q)$ to a $\EC$-special twisted cubic.
\end{theorem}

\begin{proof} As in the proof of Theorem~\ref{conictonrc}, we can without loss of generality prove this for the \orsp\ $\Bpi$ coordinatised in Theorem~\ref{thm:Bpi}, and the $(\Bpi,\li)$-special conic $\C=\sigma(\C')$ in $\Bpi$.
We continue our calculations from the end of proof of Theorem~\ref{conictonrc},  using the same notation. 
That is, let  $G(\theta)=
-(\theta-\tau^q)(\theta-\tau^{q^2})$, $f(\theta)= (\theta-\tau)(\theta-\tau^q)(\theta-\tau^{q^2})$ and $P_\theta=([G(\theta)^{q}],[G(\theta)],f(\theta))$. Then the 
 set $[\C]=\{P_\theta\st\theta\in{\mathbb F}_{q}\cup\{\infty\}\}$ is a twisted cubic which we denote by $\N=[\C]$. 
%The extension of the twisted cubic $\{(f_0(\theta),f_1(\theta),f_2(\theta),f_3(\theta))\st\theta\in{\mathbb F}_{q}\}$ of $\PG(3,q)$ to $\PG(3,q^3)$ is 
%$\{(f_0(\theta),f_1(\theta),f_2(\theta),f_3(\theta))\st\theta\in{\mathbb F}_{q^3}\}$.
We can uniquely write $G(\theta)=g_0(\theta)+g_1(\theta)\tau+g_2(\theta)\tau^2$ and 
$G(\theta)^{q}=h_0(\theta)+h_1(\theta)\tau+h_2(\theta)\tau^2$ where $g_i(\theta),h_i(\theta)$, $i=0,1,2$, are polynomials over ${\mathbb F}_{q}$. As  $f(\theta)$ is  a polynomial over ${\mathbb F}_{q}$,  the extension $\N^*$ of $\N$ to $\PG(6,q^3)$  has points $P_\theta= (h_0(\theta),h_1(\theta),h_2(\theta),\ g_0(\theta),g_1(\theta),g_2(\theta),\ f(\theta))$, for $\theta\in{\mathbb F}_{q^3}\cup\{\infty\}$.
As $f(\theta)$ has zeros $\tau,\tau^q,\tau^{q^2}\!,$\,   $\N^*$ meets $\Sigma_\infty$ in the points $P_\tau$, $P_{\tau^q}$ and $P_{\tau^{q^2}}$. 

By the proof of Theorem~\ref{conictonrc}, the 3-space containing $\N$ meets $\si$ in the cover plane  $[C_1]$ of $\EC$. 
We recall \cite[Theorem 6.3]{BJ-ext2} which shows that 
the transversal  line $\gc=\langle A_1, A_2^q\rangle$ of $\EC$ meets the (extended) cover plane $[C_1]$ in the point $A_1+\eta^{1-q}A_2^q$, where $A_1=(p_0,p_1,p_2,0,0,0,0),$ $A_2=(0,0,0,p_0,p_1,p_2,0)$, $p_0=t_1+t_2\tau-\tau^2$, $p_1=t_2-\tau$, $p_2=-1$ and $\eta=p_0+p_1\tau+p_2\tau^2$. We will  show that the  point  $P_{\tau^q}$ lies on   the transversal $\gc$. As in the proof of Theorem~\ref{conictonrc}, on simplifying, we have $G(\theta)=(-\theta^2+t_2\theta+t_1)+(-\theta+t_2)\tau-\tau^2$, so $g_0(\theta)=-\theta^2+t_2\theta+t_1$, $g_1(\theta)=-\theta+t_2$ and $g_2(\theta)=-1$.  
Letting $A=(p_0,p_1,p_2)$, and noting that 
 $g_0(\tau^q)=p_0^q$, $g_1(\tau^q)=p_1^q$, $g_2(\tau^q)=p_2^q$, we have $A_2^q=(0,0,0,g_0(\tau^q),g_1(\tau^q),g_2(\tau^q),0)=([0],\ A^q,\ 0)$.
Thus $P_\tau=([(A^q)^q],\ A^q,\ 0)$. It follows from \cite[Equation (8)]{BJ-ext2} that $(A^q)^q=\eta^{q-1} A$. Hence $P_\tau=(\eta^{q-1} A, A^q,0)=\eta^{q-1} A_1+ A_2^q\equiv A_1+\eta^{1-q} A_2^q$. Hence  $P_{\tau^q}=\gc\cap[C_1]^*=\gc\cap\N^*$. Similarly, $\N^*$ meets the transversals $\gc^q,\gc^{q^2}$, and so 
 $\N=[\C]$ is a $\EC$-special twisted cubic.
\end{proof}

The converse of this result also holds. That is, a $\EC$-special twisted cubic  corresponds  to a special conic in some exterior \orsp\ with conic cover $\EC$.

\begin{theorem}\Label{two-specials-conv}
Let $\ES$ be an exterior splash in $\PG(6,q)$, $q\geq 3$,  let $\EX$ be one cover of $\ES$, and let
 $[\N]$ be an
 $\EX$-special twisted cubic. Then in $\PG(2,q^3)$, there is a unique exterior \orsp\ $\pi$  that contains $\N$. Further, $\pi$ has exterior splash $\ES$, conic cover $\EX$, and $\N$ is a  $(\pi,\li)$-special conic  of $\pi$. 
\end{theorem}

\begin{proof} Fix an exterior splash $\ES$. 
Let $\pi$ be any exterior \orsp\ of $\PG(2,q^3)$ with exterior splash $\ES$ and conic cover $\EC$, and let $\C$ be a $(\pi,\li)$-special conic in $\pi$. By Theorem~\ref{two-specials},   in $\PG(6,q)$, $\C$ corresponds to a $\EC$-special twisted cubic. We show that the converse is true by counting the two sets and showing they have the same size. 
Let $x$ be the number of pairs $(\C,\pi)$ where $\pi$ is an exterior \orsp\  of $\PG(2,q^3)$ with the given exterior splash $\ES$, and $\C$ is a $(\pi,\li)$-special conic of  $\pi$. 
In $\PG(6,q)$, let $\EX$ be a cover of the given exterior splash $\ES$, and let $y$ be the number  of $\EX$-special twisted cubics.
Note that for any set $\N$ of $q+1$ affine points of $\PG(2,q^3)$,  if $[\N]$ is an $\EX$-special twisted cubic of $\PG(6,q)$, then  $\N$  contains a quadrangle, so $\N$ lies in at most one exterior \orsp. Further, if $\N$ does lie in an \orsp\ with exterior splash $\ES$, then by Theorem~\ref{two-specials}, $\N$ {\em may} be  a $(\pi,\li)$-special conic of $\pi$. 
That is, $x\geq y$ with equality if and only if every $\EX$-special twisted cubic corresponds to a $(\pi,\li)$-special conic in an exterior \orsp\ $\pi$ with exterior splash $\ES$. 

We first count $x$. By \cite[Theorem 4.5]{BJ-ext1}, there are $2q^6(q^3-1)$ exterior \orsps\ with the fixed exterior splash $\ES$. Further, an exterior \orsp\  contains $q^2+q+1$ special conics.
 Hence $x=2q^6(q^3-1)(q^2+q+1)$.
Now we count $y$. Note that for each cover $\EX$ of $\ES$, an $\EX$-special twisted cubic  lies in a $3$-space  that meets $\si$ in a plane $\alpha$ of the cover $\EX$.
There are $2(q^2+q+1)$ choices for  a cover plane $\alpha$, $q^3$ $3$-spaces of $\PG(6,q)\setminus\si$ through $\alpha$, and by \cite[Lemma 2.5]{BJ-iff} there are $q^3(q^3-1)$ $\EX$-special twisted cubics in such a $3$-space. 
Hence $y=2(q^2+q+1)q^6(q^3-1)$. 
Thus $x=y$, hence each $\EX$-special twisted cubic $[\N]$ (in a 3-space that meets $\si$ in a cover plane of $\EX$) corresponds to a set of points $\N$ in $\PG(2,q^3)$ that lie in a unique exterior \orsp\ $\pi$ with exterior splash $\ES$. Further, $\N$ is  a $(\pi,\li$)-special conic in $\pi$, and so by Theorem~\ref{two-specials}, $\EX$ is the conic cover of $\pi$.
\end{proof}

We now consider a non-special conic of the exterior \orsp\ $\pi$, we will show that it corresponds to a 6-dimensional normal rational curve in $\PG(6,q)$.  First we show that all conics of $\pi$ correspond to 6- or 3-dimensional normal rational curves of $\PG(6,q)$.

\begin{lemma}\Label{general-conic}
Let $\C$ be an irreducible conic of an exterior \orsp\ $\pi$.  In $\PG(6,q)$, $[\C]$ is either  a $6$-dimensional normal rational curve  or a twisted cubic. 
\end{lemma}

\begin{proof}
By Theorem~\ref{transsplashes}, we can without loss of generality prove this for the \orsp\ $\Bpi$ coordinatised in Theorem~\ref{thm:Bpi}. 
We first calculate the equation of a general irreducible conic $\C'$  in $\pi_0=\PG(2,q)$, then use the homography $\sigma$ from Theorem~\ref{thm:Bpi} to map it to a general irreducible conic $\C$ in $\Bpi$. Let $\C'$ be the image of the conic $\C''=\{
(1,\theta,\theta^2)\st\theta\in{\mathbb F}_{q}\cup\{\infty\}\}\subset\pi_0$ under a homography of $\PG(2,q)$  with matrix $A=(a_{ij})$.
The point $(1,\theta,\theta^2)\in\C''$
 is mapped  to the point $(f_0(\theta),f_1(\theta),f_2(\theta))\in\C'$, for $\theta\in{\mathbb F}_{q}\cup\{\infty\}$, where $f_0(\theta)=a_{00}+a_{01}\theta+a_{02}\theta^2$,
$f_1(\theta)=a_{10}+a_{11}\theta+a_{12}\theta^2$, $f_2(\theta)=a_{20}+a_{21}\theta+a_{22}\theta^2$,
 are polynomials over ${\mathbb F}_{q}$.
Under the homography $\sigma$ with matrix $K$ of Theorem~\ref{thm:Bpi}, this is mapped to a general irreducible conic $\C$ of $\Bpi$ with points 
\begin{eqnarray}
F_\theta=\big(-f_0(\theta)\tau+f_1(\theta),\ -f_0(\theta)\tau^q+f_1(\theta),\ f_0(\theta)\tau\tau^q-f_1(\theta)(\tau+\tau^q)+f_2(\theta)\big),\label{fdef}
\end{eqnarray}
$\theta\in{\mathbb F}_{q}\cup\{\infty\}$. We write this as $F_\theta=(\gg_0(\theta),\gg_1(\theta),\gg_2(\theta))$. 
Multiply all the coordinates of $F_\theta$ by $\gg_2(\theta)^{q^2+q}$, so $F_\theta\equiv(\gg_0(\theta)\gg_2(\theta)^{q^2+q},\gg_1(\theta)\gg_2(\theta)^{q^2+q},\gg_2(\theta)^{q^2+q+1})$, and now the last coordinate lies in ${\mathbb F}_{q}$. So in $\PG(6,q)$, 
$$[\C]=\big\{F_\theta\equiv\big([\gg_0(\theta)\gg_2(\theta)^{q^2+q}],\ [\gg_1(\theta)\gg_2(\theta)^{q^2+q}],\ \gg_2(\theta)^{q^2+q+1}\big)\st\theta\in{\mathbb F}_{q}\cup\{\infty\}\big\}.$$ 
Note that $\gg_0(\theta),\gg_1(\theta),\gg_2(\theta)$ 
have degree at most two in $\theta$, and so $\gg_2(\theta)^{q^2+q}=\gg_2(\theta)^{q^2}\gg_2(\theta)^q$ has degree at most four in $\theta$.  
Hence each coordinate of the points $F_\theta$ of $[\C]$ is a  polynomial in $\theta$ of degree at most six. Thus $[\C]$ is a rational curve in $\PG(6,q)$ of order at most 6.  We want to show that  $[\C]$ is normal in either 6-space, or in a $3$-space.

We begin by showing that $\gg_2(\theta)=f_0(\theta)\tau\tau^q-f_1(\theta)(\tau+\tau^q)+f_2(\theta)$ has degree two in $\theta$.
Expanding gives 
$\gg_2(\theta)=(a_{00}\tau\tau^q-a_{10}(\tau+\tau^q)+a_{20})+(a_{01}\tau\tau^q-a_{11}(\tau+\tau^q)+a_{21})\theta+(a_{02}\tau\tau^q-a_{12}(\tau+\tau^q)+a_{22})\theta^2.
$
If $\gg_2(\theta)$ is of degree less than two, then $a_{02}\tau\tau^q-a_{12}(\tau+\tau^q)+a_{22}=0$. Hence $a_{02}=a_{12}=a_{22}=0$ (since $1,\tau\tau^q,\tau+\tau^q$ are linearly independent over $\GF(q)$) contradicting $|A|\ne 0$.  Thus $\gg_2(\theta)$ is a polynomial of degree two in $\theta$ over ${\mathbb F}_{q^3}$.

Further, if
 $\gg_2(\theta)$ was reducible 
over ${\mathbb F}_{q}$, then let $k\in{\mathbb F}_{q}$ be a root.  Then the point $F_k$ lies in $\li$ and lies in the conic $\C$ of $\Bpi$,  contradicting $\Bpi$ being a subplane exterior to $\li$.
Hence $\gg_2(\theta)$ is irreducible over ${\mathbb F}_{q}$.

Suppose that $\gg_2(\theta)$ is irreducible over ${\mathbb F}_{q^3}$.
Then $g_0(\theta),g_1(\theta),g_2(\theta)$ have no common factors over ${\mathbb F}_{q}$. Hence the set of points $[\C]$ is a normal rational curve of order six in $\PG(6,q)$, which is the first possibility in the theorem.

Now suppose $\gg_2(\theta)$ is reducible over ${\mathbb F}_{q^3}$, so 
$\gg_2(\theta)$ has roots $\delta,\eta\in{\mathbb F}_{q^3}\setminus{\mathbb F}_{q}$. So $\gg_2(\theta)^{q^2+q+1}=r(\theta)\times s(\theta)$ where $r(\theta)$ is an irreducible cubic over ${\mathbb F}_{q}$, with roots $\delta, \delta^q,\delta^{q^2}$ over ${\mathbb F}_{q^3}$, and $s(\theta)$ is an irreducible cubic over ${\mathbb F}_{q}$ with roots $\eta,\eta^q,\eta^{q^2}$ in ${\mathbb F}_{q^3}$. 
If neither $r(\theta)$ nor $s(\theta)$ is a factor of both $\gg_0(\theta)\gg_2(\theta)^q\gg_2(\theta)^{q^2}$ and $\gg_1(\theta)\gg_2(\theta)^q\gg_2(\theta)^{q^2}$, then each coordinate has  degree six in $\theta$, and $\gg_0(\theta),\gg_1(\theta)$ and $\gg_2(\theta)$ have no common factors over ${\mathbb F}_{q}$, and hence the set of points $[\C]$ is a normal rational curve of order six in $\PG(6,q)$, again the first possibility in the statement of the theorem.

Otherwise suppose  that $r(\theta)$ say is a factor of both  $\gg_0(\theta)\gg_2(\theta)^q\gg_2(\theta)^{q^2}$ and $\gg_1(\theta)\gg_2(\theta)^q\gg_2(\theta)^{q^2}$.  As $\gg_0(\theta)$ is a polynomial in $\theta$ of degree one or two over ${\mathbb F}_{q}$, the roots of $\gg_0(\theta)$ lie in either ${\mathbb F}_{q}$ or $\GF(q^2)$, and so cannot be a root of $r(\theta)$ or $s(\theta)$, as their roots lie in ${\mathbb F}_{q^3}\backslash{\mathbb F}_{q}$.
So if $r(\theta)$ is an irreducible (over ${\mathbb F}_{q}$) factor of  $\gg_0(\theta)\gg_2(\theta)^q\gg_2(\theta)^{q^2}$, it follows that $r(\theta)$ divides $\gg_2(\theta)^q\gg_2(\theta)^{q^2}$, which has two factors (over ${\mathbb F}_{q^3}$) of $r(\theta)$, and two factors (over ${\mathbb F}_{q^3}$)  of $s(\theta)$. 
Hence either $\eta=\delta^q$ or $\delta^{q^2}$, 
and so $r(\theta)=cs(\theta)$ for some $c\in{\mathbb F}_{q^3}'$. 
Note that if $\eta=\delta^q$ we obtain $\gg_2(\theta)=c(\theta-\delta)(\theta-\delta^q)$, and if $\eta=\delta^{q^2}$ we obtain $\gg_2(\theta)=c(\theta-\delta)(\theta-\delta^{q^2})$. However, $(c(\theta-\delta)(\theta-\delta^{q^2}))^q=c^q(\theta-\delta)(\theta-\delta^q)$ so without loss of generality, we may assume that $\eta=\delta^q$. 
So in $\PG(2,q^3)$, we have $
F_\theta\equiv(\gg_0(\theta)(\theta-\delta^q)r(\theta),\ \gg_1(\theta)(\theta-\delta^q) r(\theta),\ cr(\theta)^2)
\equiv(\gg_0(\theta)(\theta-\delta^q),\ \gg_1(\theta)(\theta-\delta^q),\ cr(\theta)).$
Hence in $\PG(6,q)$, 
\begin{eqnarray}
[\C]=\big\{F_\theta=\big([\gg_0(\theta)(\theta-\delta^q)],\ [\gg_1(\theta)(\theta-\delta^q)],\ cr(\theta)\big)\st\theta\in{\mathbb F}_{q}\cup\{\infty\}\big\}\label{spec-final-eqn}.\end{eqnarray} 
That is, the rational curve of order six reduces to a rational curve of order three, so it lies in some $3$-space. If $[\C]$ is not normal in this $3$-space, then the coordinates of $F_\theta$ in 
(\ref{spec-final-eqn})
 have a common linear factor over ${\mathbb F}_{q}$, a contradiction as $r(\theta)$ is irreducible over  ${\mathbb F}_{q}$.  
Hence $[\C]$ is a normal rational curve in a $3$-space, that is, a twisted cubic. So we have the two possibilities in the statement of the theorem.
\end{proof}

As an immediate consequence of the preceding results, we have the following characterisation of the conics of an exterior \orsp\ of $\PG(2,q^3)$.

\begin{corollary}\Label{6dim-nrc} Let $\pi$ be an exterior \orsp\ of $\PG(2,q^3)$ with conic cover $\EC$, and let $\C$ be a conic of $\pi$. Then $\C$ is a $(\pi,\li)$-special conic of $\pi$  if and only if in $\PG(6,q)$, $[\C]$ is a $\EC$-special twisted cubic.
Otherwise, $[\C]$ is a 6-dimensional normal rational curve of $\PG(6,q)$. 
\end{corollary}

Finally, we characterise $\ET$-special conic of $\PG(6,q)$ as follows. Let $\pi$ be an exterior \orsp\ of $\PG(2,q^3)$ with tangent cover $\ET$, and let $[\N]$ be a $\ET$-special twisted cubic of $\PG(6,q)$. By Theorem~\ref{two-specials-conv}, there is a unique \orsp\ $\pi'$ of $\PG(2,q^3)$ that contains the points of $\N$, and further $\pi'$ has conic cover $\ET$. That is, $\N$ is a $(\pi',\li)$-special conic of $\pi'$.  Also, by  \cite[Theorem 7.2]{BJ-ext1}, $\pi$ and $\pi'$ meet in either an \orsl, or in at most three points. In either case, $\N$ has at most three points in $\pi$. 
In summary, we have:
 
 \begin{corollary}\Label{cor:t-special-tc}  Let $\pi$ be an exterior \orsp\ in $\PG(2,q^3)$ with exterior splash $\ES$, conic cover $\EC$ and tangent cover $\ET$. A  $\ET$-special twisted cubic of $\PG(6,q)$  corresponds to a set  $\C$ of $q+1$ points in $\PG(2,q^3)$, at most three of which lie in $\pi$. Further, $\C$ lies in a unique exterior \orsp\ $\pi'$, and forms a $(\pi',\li)$-special conic in $\pi'$. 
 \end{corollary}
 
We note that comments in Section~\ref{sec:mis} further support this fundamental difference in behaviour between the two covers of $\ES$ with respect to an associated exterior \orsp\ $\pi$. 

\subsection{Tangent lines to special conics in $\PG(6,q)$}\Label{sec:tgt-line-sc}

We continue our study of special conics in an exterior \orsp\ of $\PG(2,q^3)$ in this section by looking at a relationship between tangent lines of special conics in $\PG(2,q^3)$ and tangent lines of special twisted cubics in $\PG(6,q)$. It is straightforward to show that in an exterior \orsp\ $\pi$, there is a unique $(\pi,\li)$-special conic $\C$ through a fixed point $P$ tangent to a fixed line $\ell$ of $\pi$ through $P$. 
We show that in $\PG(6,q)$, the tangent lines to the two special twisted cubics $[\C]$ and $[\ell]$ at $P$ are the same.

\begin{theorem}\Label{tgtpleqconictgtpl}
Let $\pi$ be an exterior \orsp,  $P$ a point of $\pi$, and $\ell$ a line of $\pi$ through $P$. Let $\C$ be the  unique $(\pi,\li)$-special conic  of $\pi$ through $P$ with tangent $\ell$.  In $\PG(6,q)$,  the unique tangent at $P$ to the $\S$-special twisted cubic $[\ell]$, and  the unique tangent at $P$ to the $\EC$-special twisted cubic $[\C]$ are the same line.
\end{theorem}

\begin{proof} 
By Theorem~\ref{transsplashes} and \cite[Theorem 2.2]{BJ-ext1}, we can without loss of generality prove this for the exterior \orsp\ $\Bpi$ coordinatised in Theorem~\ref{thm:Bpi}, and the point $P=(0,0,1)$. 
We begin with the  \orsp\ $\pi_0=\PG(2,q)$ and the exterior line $\elltau$, and calculate the coordinates of the $q+1$ $(\pi_0,\ell$)-special conics $\C_k$ ($k\in\mathbb F_q\cup\{\infty\}$) of $\pi_0$ through the point $P=(0,0,1)$. We then apply the homography $\sigma$ from Theorem~\ref{thm:Bpi} to transform each conic $\C_k$  to a $(\Bpi,\li)$-special conic $\D_k$ of $\Bpi$ through the point $\sigma(P)=P$. 

By \cite[Lemma 2.4]{BJ-ext1},  the exterior splash $\ES$ of $\pi_0$ on $\ell$ has carriers $\car=(1,\w,\w^2)$, $\car^\r$. 
Consider the two conics $\C_0:y^2-xz=0$ and $\C_\infty:-t_0x^2+yz-t_2xz-t_1xy=0$ of $\pi_0$. They both contain the points $P,\car,\car^q,\car^{q^2}$, hence each conic in the  pencil  $\{\C_k=\C_0+k\C_\infty\st k\in{\mathbb F}_{q}\cup\infty\}$ contains the four points $P,\car,\car^q,\car^{q^2}$. That is, each $\C_k$ is a $(\pi_0,\ell)$-special conic of $\pi_0$. For $q$ both even and odd, the tangent line $t_k'$ at $P$ to $\C_k$ is the line $[-1-kt_2,k,0]=[w,1,0]$ where 
\begin{equation}\label{w-eqn}w=\frac{-kt_2-1}{k}.\end{equation} Further, the conic  $\C_k$ consists of the points 
$C_{k,\theta}=(1, \theta,f(\theta)) $ for $\theta\in{\mathbb F}_{q}\cup\{\infty\}$, where $f(\theta)=(-\theta^2+kt_1\theta+kt_0)/(k\theta-kt_2-1)$.
We now apply the homography $\sigma$ of Theorem~\ref{thm:Bpi} with matrix $K$ to map:
 the $(\pi_0,\ell)$-special conic $\C_k$ (with points $C_{k,\theta}$) of $\pi_0$  to a $(\Bpi,\li)$-special conic ${\mathcal D}_k$  (with points $D_{k,\theta}$) of $\Bpi$; and 
 the tangent line $t_k'$ to $\C_k$ at $P$  to the tangent line $t_k=\sigma(t_k')$ to $\D_k$ at $P$. In $\PG(6,q)$, $[t_k]$ is an $\ES$-special twisted cubic by Theorem~\ref{sublinesinBB}, and $[\mathcal D_k]$  is a $\EC$-special twisted cubic 
by Theorem~\ref{two-specials} (where $\EC$ is the conic cover of $\Bpi$).

Firstly, we consider the tangent line at the point $P$ to the $\ES$-special twisted cubic $[t_k]$,  we
denote the intersection of this tangent line with $\si$ by  $I_{P,{t_k}}$. %The proof of \cite[Theorem??]{BJ-ext2} calculates this  tangent line, and shows it meets $\si$ in the point 
%$I_{P,{t_k}}=([\tau],[\tau^q],0)-w([1],[1],0)
%$.
Note that points of $\pi_0$ on the line $t_k'$  distinct from $P$ have coordinates $P_x'=(1,-w,x)$ for $x\in{\mathbb F}_{q}$, so points  of $\Bpi$ on the line $t_k$ distinct from $P$ have coordinates $P_x=\sigma(P_x')=(-\tau-w,\ -\tau^q-w,\ \tau\tau^q+(\tau+\tau^q)w+x)$, for $x\in{\mathbb F}_{q}$.
To convert this to a coordinate in $\PG(6,q)$, we need to multiply by an element of ${\mathbb F}_{q^3}$ so that the last coordinate lies in ${\mathbb F}_{q}$. Let $F(x)=\tau\tau^q+(\tau+\tau^q)m+x$ (the third coordinate in $P_x$). As $F(x)\in{\mathbb F}_{q^3}$, we have $F(x)^{q^2+q+1}\in{\mathbb F}_{q}$, so multiply the coordinates of $P_x$ by $F(x)^{q^2+q}$. 
So the point $P_x\in\Bpi$ corresponds in $\PG(6,q)$ to the point $P_x$ with coordinates 
$
P_x= ([-(\tau+w)F(x)^{q^2+q}],\ [-(\tau^q+w)F(x)^{q^2+q}],\ F(x)^{q^2+q+1})
$.
To calculate $I_{P,{t_k}}$, let $Q_x=PP_x\cap\Sigma_\infty$. Now $F(x)^{q^2+q}$ is equal to $x^2$ plus lower powers of $x$. As $P=(0,0,1)$ is $P_x$ with $x=\infty$,  we can calculate $I_{P,{t_k}}=Q_\infty$ by dividing all the coordinates by $x^2$ and letting $x\rightarrow\infty$.
\begin{eqnarray*}
%I_{P,{t_k}}&=&\lim_{x\rightarrow\infty} PP_x\cap\Sigma_\infty
%=\lim_{x\rightarrow\infty} ([-(\tau+w)F(x)^{q^2+q}],[-(\tau^q+w)F(x)^{q^2+q}],0)\\
%&=&([-(\tau+w)],\ -[\tau^q+w], 0)\\
%&\equiv& w([1],[1],0)+([\tau],[\tau^q],0).
I_{P,{t_k}}&=&\lim_{x\rightarrow\infty} PP_x\cap\Sigma_\infty
=([-\tau-w], [-\tau^q-w], 0)
\ \equiv\  w([1],[1],0)+([\tau],[\tau^q],0).
\end{eqnarray*}

Secondly, we consider the tangent line at the point $P$ to the $\EC$-special twisted cubic $[\mathcal D_k]$, we denote the intersection of this tangent line with $\si$ by  $I_{P,{\D_k}}$.
The conic $\D_k$ has points $
D_{k,\theta}=\sigma(\C_{k,\theta})=(
-\tau+\theta,\ -\tau^q+\theta,\ \tau\tau^q-(\tau+\tau^q)\theta+f(\theta))$.
Let $g(\theta)=\tau\tau^q-(\tau+\tau^q)\theta+f(\theta)$ (the third coordinate), and multiply all coordinates by $g(\theta)^{q^2+q}$, so that the third coordinate is now in ${\mathbb F}_{q}$.
 Then in $\PG(6,q)$, $[\D_k]$ has points $D_{k,\theta}= \big([(-\tau+\theta)g(\theta)^qg(\theta)^{q^2}],\ [(-\tau^q+\theta)g(\theta)^qg(\theta)^{q^2}],\ g(\theta)g(\theta)^qg(\theta)^{q^2}\big)$, for $\theta\in{\mathbb F}_{q}\cup\{\infty\}$. By considering $f(\theta)$, we see that $\theta=(kt_2+1)/k$ corresponds to the point $P=(0,0,1)$. 
 
 To calculate $I_{P,\mathcal D_k}$, we 
 re-parameterise the points $D_{k,\theta}$ so that the point $P$ of $[\D_k]$  has parameter $\infty$. Replace the parameter $\theta\in{\mathbb F}_{q}\cup\{\infty\}$ by the parameter $r\in{\mathbb F}_{q}\cup\{\infty\}$ where $\theta=1/r+(kt_2+1)/k=1/r-w$ with $w$ as in (\ref{w-eqn}) (equivalently, $r=1/(\theta+w)$). Using the parameter $r$, the twisted cubic $[\D_k]$ has points we denote by $D_{k,r}$, $r\in\mathbb F_q\cup\{\infty\}$, and we have $P$ is the point $D_{k,\infty}$. 
 For a fixed $k$, the line joining $P=D_{k,\infty}$ to $D_{k,r}$ meets $\si$ in a point denoted $A_r$, $r\in{\mathbb F}_{q}$. We let $r\rightarrow\infty$ to find the point $I_{P,\mathcal D_k}=A_\infty$.
As we will only be interested in the coordinates for $D_{k,r}$ in the case 
when  $r\rightarrow\infty$,  we look at the  coordinates  of $D_{k,r}$ as polynomials in $r$, and only calculate the highest powers of $r$.
We can write $g(\theta)=ar+\mu+\nu/r$, where $a=w^2/k-t_1w+t_0\in{\mathbb F}_{q}$, and $\mu,\nu\in{\mathbb F}_{q^3}$. 
Then
the first coordinate of $D_{k,r}$  is $-(\tau+w)a^2r^2$ plus lower powers of $r$.  The second coordinate of  $D_{k,r}$ is $-(\tau^q+w)a^2r^2$ plus lower powers of $r$.  The third coordinate of $D_{k,r}$ is $a^3r^3$ plus lower powers of $r$.
Thus
$
I_{P,\mathcal D_k}=\lim_{r\rightarrow\infty}PD_{k,r}\cap\si
=\left(\left[-(\tau+w\right)a^2],\ \left[-(\tau^q+w)a^2\right],\ 0\right)
\equiv([\tau],[\tau^q],0)+w([1],[1],0),
$
as $a\in{\mathbb F}_{q}$.
This is the same as the point $I_{P,{t_k}}$ calculated above, hence the two tangent lines in $\PG(6,q)$ through $P$ are equal.
\end{proof}

%%%%%%%%%%
%%%%%%%%%%

\subsection{Consequences of the special conic representation}\Label{sec:mis}

The results of Section~\ref{sec:sp-con} about the representation in $\PG(6,q)$ of $(\pi,\li)$-special conics of an exterior \orsp\ $\pi$ have two interesting applications. 
Firstly, we consider the structure of an exterior \orsp\ $\pi$ in $\PG(6,q)$. We showed in \cite[Theorem 4.1]{BJ-ext2} that $[\pi]$ is the intersection of nine quadrics. We show in Theorem~\ref{thm:quad-trans} that each of these nine  quadrics contains the transversal lines of both the exterior splash $\ES$  and the conic cover $\EC$ of $\pi$. 
Secondly, in Theorem~\ref{pipidash}, we consider applications to replacement sets of $\PG(5,q)$.

\begin{theorem} \Label{thm:quad-trans}
Let $\pi$ be an exterior \orsp\ of $\PG(2,q^3)$ with exterior splash $\ES$ and conic cover $\EC$. Then in the cubic extension $\PG(6,q^3)$, the nine quadrics that determine $[\pi]$ each contain the transversal lines  $\gs,\gs^q,\gs^{q^2}$, $\gc,\gc^q,\gc^{q^2}$ of $\ES$ and $\EC$.\end{theorem}

\begin{proof}
  By Theorem~\ref{sublinesinBB}, 
  the $q^2+q+1$ lines of $\pi$ correspond to $q^2+q+1$ $\ES$-special twisted cubics in $[\pi]$, each lying in a 3-space through a distinct plane of $\ES$. So  
 in $\PG(6,q^3)$, each of the nine quadrics defining $[\pi]$ contains $q^2+q+1$ points on $\gs$.  As 
$q^2+q+1>3$, each of these nine quadrics contains the transversal line $\gs$, and hence also contains its conjugates $\gs^q,\gs^{q^2}$. 

Next we consider the $q^2+q+1$ $(\pi,\li)$-special conics in $\pi$. By Theorem~\ref{two-specials} and Corollary~\ref{s-con-in-distinct}, these correspond to $q^2+q+1$ $\EC$-special conics, each containing a distinct point of $g_\EC$. As $q^2+q+1>3$, each of the nine quadrics containing $[\pi]$ contain $g_\EC$ and similarly contain $g_\EC^q$ and $\gc^{q^2}$.
\end{proof}

We conjecture that none of the nine quadrics defining $[\pi]$ contain the transversal lines  $\gt,\gt^q,\gt^{q^2}$ of the tangent cover $\ET$, but a geometric proof eludes us.

Recall that in $\PG(6,q)$, $\S$ is a 
 regular 2-spread, and the Bruck-Bose plane associated with $\S$ is denoted $\P(\S)\cong\PG(2,q^3)$. Let $\pi$ be an exterior \orsp\ of $\P(\S)$ with exterior splash $\ES\subset \S$ and conic cover $\EC$. 
In \cite{BJ-ext2}, we showed that the conic cover $\EC$ is contained in a unique regular 2-spread (with transversal lines $\gc, \gc^q,\gc^{q^2}$) denoted $\S_\EC$.
The regular 2-spread $\S_\EC$ gives rise to the Bruck-Bose plane $\P(\S_\EC)\cong\PG(2,q^3)$. 
We consider the representation of $\pi$  (and its lines and special conics) in the Bruck-Bose plane $\P(\S_\EC)$.

\begin{theorem}\Label{pipidash}
Let $\pi$ be an exterior \orsp\ in $\P(\S)\cong\PG(2,q^3)$ with  exterior splash $\ES$, conic cover $\EC$, and tangent cover $\ET$. The points of $[\pi]$ correspond to a set of points $\pi'$ in the Desarguesian plane $\P(\S_\EC)$ satisfying: 
\begin{enumerate}
\item 
$\pi'$ is an exterior \orsp\ of $\P(\S_\EC)$,
\item the lines of $\pi$ correspond exactly to the $(\pi',\li'$)-special conics of $\pi'$, 
\item the $(\pi,\li)$-special conics of $\pi$ correspond exactly to the lines of $\pi'$,
\item $\pi'$ has exterior splash  $\EC$, conic cover   $\ES$, and tangent cover $\ET$.
\end{enumerate}
\end{theorem}

\begin{proof}
Let $\pi$ be an exterior \orsp\ in $\PG(2,q^3)$ with exterior splash $\ES$, conic cover $\EC$ and tangent cover $\ET$. 
Label the lines of $\pi$ in $\PG(2,q^3)$ by $\ell_1,\ldots,\ell_{q^2+q+1}$, and label the  $q^2+q+1$
 $(\pi,\li)$-special conics in $\pi$ by $\C_1,\ldots,\C_{q^2+q+1}$.
 By Theorem~\ref{sublinesinBB}, the line $\ell_i$ corresponds to an $\S$-special twisted cubic $[\ell_i]$ in a $3$-space about a plane of $\ES$. By Theorem~\ref{two-specials}, the $(\pi,\li)$-special conic $\C_i$ corresponds to a $\EC$-special twisted cubic  in a $3$-space about a plane of $\EC$. 
By Theorem~\ref{sublinesinBB},
in the Desarguesian plane $\P(\S_\EC)$, $\pi'$ is a set of $q^2+q+1$ points, 
with lines $\C_1,\ldots,\C_{q^2+q+1}$. As the points and $(\pi,\li)$-special conics of $\pi$ form a Desarguesian plane,   $\pi'$ is an exterior \orsp\ of $\P(\S_\EC)$.
Further by Theorem~\ref{two-specials-conv}, $\pi'$ has $(\pi',\li')$-special conics $\ell_1,\ldots,\ell_{q^2+q+1}$.
Moreover  in $\PG(6,q)$, $\pi'$ has exterior splash  $\EC$, conic cover $\ES$, and tangent cover $\ET$.
\end{proof}

We relate this theorem to the nine quadrics that define $[\pi]$ discussed above. In $\PG(6,q)$, the sets $[\pi]$, $[\pi']$ of Theorem~\ref{pipidash} are equal,
so the  nine quadrics used to define $[\pi]$ are the same as the nine quadrics used to define $[\pi']$.  This is consistent with  Theorem~\ref{thm:quad-trans} which shows  that in $\PG(6,q^3)$, the nine quadrics defining $[\pi]$ all contain $\gs,\gs^q,\gs^{q^2}$ and $\gc,\gc^q,\gc^{q^2}$, and the  roles of $\ES$ and $\EC$ are reversed in $[\pi]$ and $[\pi']$ by Theorem~\ref{pipidash}.

In a similar way to Theorem~\ref{pipidash},  we can we  consider the unique 2-spread $\S_\ET$ (with transversal lines $\gt,\gt^q,\gt^{q^2}$) containing the tangent cover $\ET$ of $\pi$. In the Bruck-Bose plane $\P(\S_\ET)$, consider the set of points $\pi''
$ corresponding to the points of $[\pi]$. Then $\pi''$ is a set of $q^2+q+1$ points. To define `lines' of $\pi''$, we consider  the pointsets in $\pi''$ that correspond to $\ET$-special twisted cubics. By Corollary~\ref{cor:t-special-tc}, these `lines' have at most three points. So $\pi''$ is not an \orsp\ of $\P(\S_\ET)$.

\section{Order-$q$-subplanes with a common splash}

The article \cite{BJ-ext1} briefly looked at \orsps\ that share a common splash. In particular, \cite[Theorem 7.2]{BJ-ext1} showed that two \orsps\ with a common splash either share a common \orsl, or meet in at most three points.  Further (see
 Theorem~\ref{two-subplanes})  there are exactly two \orsps\ with a common exterior splash and a common \orsl. 
In this section, we further the investigation of exterior \orsps\ with a common splash. 

Let $\pi_1,\pi_2$ be the two exterior \orsps\ with a common exterior  splash $\ES$ and common \orsl\ $b$. In Section~\ref{sec:orsl-common-splash}, we investigate the two families of \orsls\ in $\ES$ in relation to $\pi_1$ and $\pi_2$. Further, we show that  if $\pi_1$ has conic cover $\EC$ and tangent cover $\ET$, then $\pi_2$ has conic cover $\ET$ and tangent cover $\EC$. 
In Section~\ref{sec:two-orsp-sp-con}, we give a relationship between the $(\pi_1,\li)$-special conics of $\pi_1$ and the $(\pi_2,\li)$-special conics of $\pi_2$.
 In Section~\ref{sec:construct}, we start with an exterior  splash $\ES$ and appropriate \orsl\ $b$, and give a geometric construction in the $\PG(6,q)$ setting of the two \orsps\ that share $\ES$ and $b$.  Note that this is analogous to the construction in \cite{BJ-tgt1} which gives a similar construction in the case of an \orsp\ tangent to $\li$ with a fixed tangent splash on $\li$.

\subsection{Order-$q$-sublines in a common splash}\Label{sec:orsl-common-splash}

 By \cite{lavr10}, an exterior splash $\ES$ of $\li$ in $\PG(2,q^3)$ contains  $2(q^2+q+1)$ \orsls\ that lie in two distinct {\em families} of size $q^2+q+1$. If $\pi$ is an \orsp\ with exterior splash $\ES$, then the two families of $\ES$ are characterised in relation to $\pi$  in \cite[Theorem 5.2]{BJ-ext1}. If $P$ is a point of $\pi$, then the lines of $\pi$ through $P$ meet $\ES$ in an \orsl\ called a {\em $\pi$-\psline} of $\ES$. The other family of \orsls\ of $\ES$ can be constructed from ``special'' dual conics in $\pi$ and these \orsls\ are called {\em $\pi$-\dcsline s} of $\ES$. Note also that by \cite[Theorem 5.3]{BJ-ext1}, there is some other exterior \orsp\ $\pi'$ with common exterior splash $\ES$ for which the $\pi$-\psline s of $\ES$ are the $\pi'$-\dcsline s of $\ES$. That is, the classification of the two families {\em depends} on the associated \orsp.
 In this section we improve this result by finding the relationship between the two \orsps\ $\pi,\pi'$. We show that the  two \orsps\ with common exterior splash $\ES$ and sharing an \orsl\ $b$ (which exist by Theorem~\ref{two-subplanes}) interchange the roles of the two families of \orsls\ of $\ES$.

\begin{theorem}\Label{interchange-sl}
Let  $\pi_1,\pi_2$ be the two exterior \orsps\ of $\PG(2,q^3)$ with common exterior splash $\ES$ and a common \orsl\ $b$. Then the $\pi_1$-\psline s of $\ES$ are the $\pi_2$-\dcsline s of $\ES$ (and 
the $\pi_1$-\dcsline s of $\ES$ are the $\pi_2$-\psline s of $\ES$).
\end{theorem}

\begin{proof}
By Theorem~\ref{transsplashes}, we can 
without loss of generality assume one of the \orsps\ is 
$\Bpi$. Recall that in $\PG(2,q^3)$, we construct $\Bpi$ from $\pi_0=\PG(2,q)$ by applying the homography of Theorem~\ref{thm:Bpi} with matrix $K$. 
We construct a second \orsp\ using 
 the homography $\psi$ of $\PG(2,q^3)$ with matrix
 $K'=HK$, where
$$H= \begin{pmatrix}
0&1&0\\ 1&0&0\\ 0&0&1
\end{pmatrix},
\quad{\rm and\ so\ }\quad
K'=\begin{pmatrix}-\tau^q&1&0\\ -\tau&1&0\\ \tau\tau^q&-\tau-\tau^q&1\end{pmatrix}. $$
The collineation $\psi$ maps $\pi_0$ to an \orsp\ $\pi$. Further $\Bpi$ and $\pi$ have the same  exterior splash $\ES$ on $\li$.
Also note that both $\Bpi$ and $\pi$ share the \orsl\ $\{K(0,y,z)^t\st y,z\in\mathbb F_q\}$.  Finally the point with coordinates $K(1,0,0)^t=(-\tau,-\tau^q,\tau\tau^q)^t$ lies in $\Bpi$ but not in $\pi$, so $\Bpi\neq \pi$. 

Note that the homography $\zeta$ with matrix $H$ interchanges $\Bpi$ and $\pi$. Further, by 
the proof of \cite[Theorem 5.3]{BJ-ext1}, $\zeta$ fixes $\ES$ and interchanges the two families of \orsls\ of $\ES$. Hence the $\Bpi$-\psline s of $\ES$ are the $\pi$-\dcsline s of $\ES$, and conversely.
\end{proof}

We now show that if $\pi_1$, $\pi_2$ are the two distinct exterior \orsps\ with common exterior splash and common \orsl, then
the roles of the conic cover and tangent cover are interchanged for $\pi_1$ and $\pi_2$.

\begin{theorem}\Label{two-covers-conics} 
Let $\pi_1,\pi_2$ be the two exterior \orsps\ of $\PG(2,q^3)$ with a common exterior splash $\ES$  and containing a common \orsl\ $b$.
Then the conic cover of $\pi_1$ is the tangent cover of $\pi_2$ (and conversely, the tangent cover of $\pi_1$ is the conic cover of $\pi_2$).
\end{theorem}

\begin{proof} 
By Theorem~\ref{transsplashes},  we can without loss of generality prove this for the two \orsps\ $\Bpi,\pi$ used in the proof of Theorem~\ref{interchange-sl}.
Label the conic cover and tangent cover of $\Bpi$ by $\EC_\Bpi$ and $\ET_\Bpi$ respectively, and the conic cover and tangent cover of $\pi$ by $\EC_\pi$ and $\ET_\pi$ respectively.

By \cite[Theorem 2.2]{BJ-ext2}, we can without loss of generality prove the result for the point $P=(0,0,1)$, which lies in the common \orsl\ $\{K(0,y,z)^t\st y,z\in\mathbb F_q\}$ of $\Bpi$ and $\pi$.  The tangent plane $\TP$ at $P=(0,0,1)$ is defined in \cite[Theorem 4.2]{BJ-ext2}, it is a plane  in $\PG(6,q)$  through $P$ that meets $\si$ in a line $\ell=\TP\cap\si$ lying in a cover plane $\alpha$ of the tangent cover $\ET_\Bpi$ of $\Bpi$. 
By \cite[Theorem 7.2]{BJ-ext2},  the line $\ell$ lies in a unique 2-regulus of $\ES$ which corresponds to an \orsl\ $d$ of $\ES\subset\li$ in $\PG(2,q^3)$. Further, $d$ is the $\Bpi$-\psline\ of $\ES$  arising from the intersection of  the lines of $\Bpi$ through $P$ with $\li$.

Now consider the homography $\zeta$ with matrix $H$ defined in the proof of Theorem~\ref{interchange-sl}. As $\zeta$  interchanges the two families of \orsls\ of $\ES$, $d'=\zeta(d)$ is a $\Bpi$-\dcsline\ of $\ES$. Hence by 
\cite[Theorem 7.2]{BJ-ext2}, in $\PG(6,q)$, $[\zeta]$ maps  the cover plane $\alpha$ of $\ET_\Bpi$ to a cover plane $\beta$ of $\EC_\Bpi$.
Further, by Theorem~\ref{interchange-sl},  $d'=\zeta(d)$ is  
 a $\pi$-\psline\ of $\ES$. Note that  $\zeta$ fixes $P$,   so  $d'$ is the \orsl\  arising from the intersection of  the lines of $\pi$ through $P$ with $\li$. 
Hence by 
\cite[Theorem 7.2]{BJ-ext2}, $\beta$ is a plane of the tangent cover $\ET_\pi$ of $\pi$. 
That is, $\EC_\Bpi=\ET_\pi$, as required. 
 \end{proof}

\subsection{Special conics of the two subplanes}\Label{sec:two-orsp-sp-con}

Let $\pi_1,\pi_2$ be the  two exterior \orsps\ of $\PG(2,q^3)$ with common exterior splash $\ES$, and a common \orsl\ $b$ (recall these exist by Theorem~\ref{two-subplanes}).  
There is no obvious relationship between the \orsls\ of $\pi_1$ and the \orsls\ of $\pi_2$, although we
 can pair them up by matching \orsls\ from each subplane that share a common point of $\ES$.
However, there is a nice geometric relationship between certain special conics of $\pi_1$ and  $\pi_2$. 
In particular, there are $q+1$ conics in $\PG(2,q^3)$ that meet both $\pi_1$ and $\pi_2$ in special conics. 

\begin{theorem}\Label{unique-sp-conic-tgt}
Let $\pi_1$, $\pi_2$ be the two exterior  \orsps\ with common exterior splash  and a common \orsl\ $b$. 
Let $\C$ be a $(\pi_1,\li)$-special conic of $\pi_1$ that is tangent to $b$. Then the extension of $\C$ to $\PG(2,q^3)$ meets $\pi_2$ in a $(\pi_2,\li)$-special conic.
\end{theorem}

\begin{proof} Let $\pi_1,\pi_2$ be two exterior \orsps\ with common exterior splash $\ES$ and \orsl\ $b$ (these exist by Theorem~\ref{two-subplanes}).  
Let 
 $\ES$ have carriers $\car_1,\car_2$, and third conjugate point $\car_3$.
For each point $P_i\in b$, $i=1,\ldots,q+1$, there is a unique conic $\C_i$ of $\PG(2,q^3)$ that contains  $\car_1,\car_2,\car_3$ and $P_i$, and is tangent to $b$ at the point $P_i$.
The conic $\C_i$ meets the \orsp\ $\pi_1$ (respectively $\pi_2$) in a conic that contains $P_i$,  is tangent to $b$ at $P_i$, and contains $\car_1,\car_2,\car_3$ (in the extension to $\PG(2,q^3)$). 
Hence for $i=1,\ldots,q+1$, $\C_i\cap\pi_1$ is a $(\pi_1,\li)$-special conic of $\pi_1$, and $\C_i\cap\pi_2$ is a $(\pi_2,\li)$-special conic of $\pi_2$.
\end{proof}

\subsection{Construction of two subplanes in $\PG(6,\r)$}
\Label{sec:construct}

In $\PG(2,q^3)$, let 
 $\ES$ be an exterior splash of $\li$, and let $b$ be an exterior \orsl\ whose extension to $\PG(2,q^3)$ meets $\ES$. By Theorem~\ref{two-subplanes}, there are exactly two distinct exterior \orsp s  $\pi_1,\pi_2$ with exterior splash $\ES$ containing $b$.  
  We now give a geometric construction in the $\PG(6,q)$ setting of these  two \orsps\ $\pi_1,\pi_2$.  Note that by  Theorem~\ref{two-covers-conics}, one of the covers of $\ES$ acts as the conic cover for $\pi_1$, and the other cover of $\ES$ acts as the conic cover for $\pi_2$. 
The next construction uses  each  set of cover planes of $\ES$ to construct one of the \orsps\ containing $b$ with exterior splash $\ES$.

\begin{construction}
In $\PG(2,q^3)$, $q\geq3$, let 
 $\ES$ be an exterior splash of $\li$, let $\ell$ be a line through a point $L$ of\, $\ES$, and let $b=\{P_1,\ldots,P_{q+1}\}$ be an \orsl\ of $\ell$ that is exterior to $\li$. Let $\mathbb{X,Y}$ be the two covers of\, $\ES$. In $\PG(6,q)$, for $i,j\in\{1,\ldots,q+1\}$, let
\begin{enumerate}
\item 
$L_{ij}=P_iP_j\cap[L]$, $i\neq j$, 
\item  $L_{ii}=m_i\cap[L]$, where $m_i$ is the  unique tangent to the $\ES$-special twisted cubic $[b]$ at the point $P_i$, 
\item  $\alpha_{ij}$ be the unique cover plane of the cover $\EX$ through the point $L_{ij}$, 
\\ $\beta_{ij}$ be the unique cover plane of the cover $\EY$ through the point $L_{ij}$, 
\item $\Sigma_{ij}=\langle\alpha_{ij},P_i\rangle$, $\Gamma_{ij}=\langle\beta_{ij},P_i\rangle$.
\end{enumerate}
Then the two exterior \orsps\ $\pi_1,\pi_2$ with common exterior splash $\ES$ and \orsl\ $b$ are constructed as follows.
\begin{enumerate}
\item[I.]
 $[\pi_1]$ consists of the points  $\Sigma_{ij}\cap\Sigma_{kl}$ for $i,j,k,l\in\{1,\ldots,q+1\}$, $\{i,j\}\ne\{k,l\}$. 
\item[II.]
 $[\pi_2]$ consists of the points  $\Gamma_{ij}\cap\Gamma_{kl}$ for $i,j,k,l\in\{1,\ldots,q+1\}$, $\{i,j\}\ne\{k,l\}$. 
 \end{enumerate}
\end{construction}

\begin{proof} 
Let $\pi_1,\pi_2$ be the two \orsps\ with exterior splash $\ES$ containing $b$ (which exist by   
Theorem~\ref{two-subplanes}).  
 By Theorem~\ref{two-covers-conics}, one of the covers, $\mathbb X$ say, acts as the conic cover for $\pi_1$, and the other cover $\mathbb Y$ acts as the conic cover for $\pi_2$. 
 We prove the construction for $[\pi_1]$, and $[\pi_2]$ is similar. 
Note that the $q+1+{{q+1}\choose2}$ points $L_{ij}$, $1\leq i\leq j\leq q+1$,  are distinct by \cite[21.1.9]{hirs85}.
Hence the $q+1+{{q+1}\choose2}$  $3$-spaces $\Sigma_{ij}$, $1\leq i\leq j\leq q+1$,   are distinct. 

We first show that each $3$-space $\Sigma_{ij}$ meets $[\pi_1]$ in a set of points that corresponds to a $(\pi_1,\li)$-special conic. 
Fix $i$, and  let $j\in\{1,\ldots,q+1\}$, and firstly suppose $j\neq i$. In $\PG(2,q^3)$, the two points $P_i,P_j$ of $b$ lie in a unique $(\pi_1,\li)$-special conic of $\pi_1$, denoted $\C_{ij}$.
In $\PG(6,q)$, $[\C_{ij}]$ is an $\EX$-special twisted cubic in a $3$-space $\Pi_{ij}$
that contains the line $P_iP_j$, and a plane of $\EX$ (the conic cover of $\pi_1$), so $\Pi_{ij}$ contains the plane $\alpha_{ij}$, that is, $\Pi_{ij}=\Sigma_{ij}$.
Now consider the case $j=i$. 
For each point $P_i\in b\subset\PG(2,q^3)$, there is a unique $(\pi_1,\li)$-special conic $\C_i$ of $\pi_1$ that is tangent to $b$ at the point $P_i$. 
In $\PG(6,q)$, $[\C_i]$ is an $\EX$-special twisted cubic in a $3$-space $\Pi_i$ about a plane of $\EX$. By Theorem~\ref{tgtpleqconictgtpl},  the line $P_iL_{ii}$ is in the $3$-space $\Pi_i$, hence $\Pi_i=\Sigma_{ii}$.
Hence for each $i,j$, $\Sigma_{ij}$ meets $[\pi_1]$ in a set of points that corresponds to a $(\pi_1,\li)$-special conic which contains one or two points of $b$.

We now show that every point in $\pi_1$ lies on at least two $(\pi_1,\li)$-special conics that meet $b$. 
This is clearly true for each point $P_i\in b$.
Let $P$ be a point of $\pi_1\setminus b$. For each point $P_i\in b$, there is a unique $(\pi_1,\li)$-special conic through $P$ and $P_i$. These special conics are not necessarily distinct, but there are at least $(q+1)/2$ distinct ones. As $q\geq 3$, the point $P$ lies on at least two $(\pi_1,\li)$-special conics that meet $b$. Hence in $\PG(6,q)$, the point $P$ lies in at least two of the $3$-spaces $\Sigma_{ij}$. 
Note that the $\Sigma_{ij}$ pairwise meet in a point, hence by determining all the intersections  $\Sigma_{ij}\cap\Sigma_{kl}$,  we construct all the points of $\pi_1$. 
\end{proof}

We note that there does not appear to be a corresponding construction in $\PG(2,q^3)$, as it is difficult to distinguish between the points of $\pi_1$ and $\pi_2$. 
With the aim of producing a simpler construction of the exterior \orsp, we very briefly discuss another way of characterising the \orsps\ in $\PG(6,q)$.
This time we work in the quadratic extension $\PG(6,q^2)$ of $\PG(6,q)$, so that we can use the 
 {\em chords} of a twisted cubic (see \cite{hirs85}).
  For a point $P\in\PG(6,q^2)$, let $P^\sigma$ denote the conjugate of $P$ arising from the Frobenius automorphism $\sigma(x)= x^q$ in $\GF(q^2)$.
Let $\N$ be a twisted cubic in a 3-space $\Sigma$. There are three types of chords of $\N$ in $\Sigma$: 
  real chords $PQ$  for $P,Q\in\N$, $P\neq Q$; imaginary chords  $PP^\sigma$ for $P\in\N^*\setminus\N$; and one tangent chord  for each $P\in\N$. Note that every point of $\Sigma$ lies on a unique chord of $\N$, see \cite{hirs85}.

Let $b$ be an exterior \orsl\ of $\PG(2,q^3)$, so by  Theorem~\ref{sublinesinBB}, $\N=[b]$ is an $\S$-special twisted cubic in a 3-space about a plane $\alpha\in\S$. 
We can uniquely label the points of the plane $\alpha$ using the $q^2+q+1$ \emph{pairs} denoted $(P,Q)$, $(P,P^\sigma)$, $(P,b)$ as follows.  A point $R\in\alpha$ lies on a unique chord of $\N$. We label $R$ by  the pair: $(P,Q)$ if $R$ lies on the real chord $PQ$ of $\N$; 
$(P,P^\sigma)$ if $R$ lies on the imaginary chord $PP^{\sigma}$ of $\N$; and $(P,b)$ if $R$ lies on the tangent  $b$ to $\N$ through $P$.
Hence if $\pi$ is an \orsp\ of $\PG(2,q^3)$ with exterior splash $\ES$, then we have a unique labelling in $\PG(6,q)$ for each point in each plane of $\ES$.

We consider the interpretation of this pair labelling in planes of
the covers $\EC$ and $\ET$ of $\pi$. 
The interpretation in planes of  the conic cover $\EC$  is straightforward: Let $\beta\in\EC$, then by the results of Section~\ref{sec:sp-con}, there is a unique $(\pi,\li)$-special conic $\C$ in $\pi$  corresponding to $\beta$. The points in $\beta$ are  labelled by pairs as follows. 1. Pairs $(P,Q)$ for $P,Q\in\C$. 2. Pairs $(P,P^\sigma)$ for points $P\in\C^*\setminus\C$. 3.  Pairs  $(P,\ell)$ for $P\in\C$ and $\ell$ a line of $\pi$ through $P$. Note that  Theorem~\ref{tgtpleqconictgtpl} ensures that the labelling in 3 works. 
The interpretation in planes of  the tangent cover $\ET$ is  more difficult, and we leave this as an open question.

%%
%%%%%%%%%%
%%%%%%%%%%
%%%%%%%%%%
%%%%%%%%%%
%%%%%%%%%%
%%%%%%%%%%
%%%%%%%%%%
%%%%%%%%%%
%%%%%%%%%%
%%%%%%%%%%
%%%%%%%%%%
%%%%%%%%%%

\section{Intersection of two exterior splashes}\Label{sec:int-ext}

In this section we investigate the intersection of two exterior splashes. First recall from Section~\ref{sec:orsl-common-splash} that an exterior splash $\ES$ of $\li$ in $\PG(2,q^3)$ contains  $2(q^2+q+1)$ \orsls\ that lie in two distinct {\em families} of size $q^2+q+1$. Note that two distinct  \orsls\ in the same family meet in exactly one point, and \orsls\ in different families meet in 0,1 or 2 points. 

Ferret and Storme 
 \cite[Lemma 2.3]{ferr03} showed that two exterior splashes meet in at most $2q+2$ points. Lavrauw and Van de Voorde \cite[Remark 24]{lavr10} show that this bound is tight by exhibiting an example of two exterior splashes who intersection has size $2q+2$, this intersection  consists of two disjoint \orsls, one from each family. 

In this section, we show in Theorem~\ref{two-orsls} that two exterior splashes cannot meet in two \orsls\ of the same family. Then in Theorem~\ref{max-int-ext}, we characterise the maximal intersection of two exterior splashes by showing that if the intersection has size $2q+2$, then it consists of two disjoint \orsls, one from each family.

We need the next lemma which follows  from \cite[Theorem 7.2]{BJ-ext2}.

\begin{lemma}\Label{L:prop-cover}
Let $\R$ be a $2$-regulus contained in an exterior splash $\ES$.
\begin{enumerate}
\item Each of the $q^2+q+1$ ruling lines of $\R$ lie in a unique cover plane of $\ES$, further these $q^2+q+1$ cover planes are in the same cover of $\ES$.
\item Each line  in a cover  plane of $\ES$ 
 is the ruling line for one $2$-regulus contained in $\ES$. Further, the $q^2+q+1$  $2$-reguli constructed from the lines in one cover plane correspond to  the $q^2+q+1$ \orsls\ in one family of $\ES$.
 \item Let $\ell,m$ be ruling lines for distinct $2$-reguli contained $\ES$. If $\ell$ and $m$ meet, then $\langle \ell,m\rangle$ is a cover plane of $\ES$. 
\end{enumerate}
\end{lemma}

\begin{theorem} \Label{two-orsls} 
Let $\ES$ be an exterior splash of $\li$ in $\PG(2,q^3)$, and let $b_1,b_2$ be \orsls\ of $\ES$ that belong to the same family. 
If $\ES'$ is any exterior splash containing $b_1$ and $b_2$, then $\ES'=\ES$.
\end{theorem}

\begin{proof} Let $b_1$ and $b_2$ be distinct \orsls\ of $\ES$ in the same family, so $b_1$ and $b_2$ meet in a unique point $P$.   By Theorem~\ref{sublinesinBB}, in $\PG(6,q)$,  $[b_1]$, $[b_2]$ are  2-reguli contained in the regular 2-spread $\S$. Let $Q$ be any point of the spread plane $[P]$, then there is a ruling line $t$ of $[b_1]$ through $Q$, and a ruling line $u$ of $[b_2]$ through $Q$.
By Lemma~\ref{L:prop-cover},  $\langle t,u\rangle$ is a cover plane of $\ES$. 
 Hence the planes of the regular 2-spread $\S$ that meet $\langle t,u\rangle$ are exactly the planes of $\ES$, that is, $\ES$ is completely determined.
\end{proof}

%
%
%
%\green{
%We will be interested in how a $3$-space of $\PG(5,q)$ meets a $2$-regulus, the next result result is from \cite[Lemma 7]{lavr10}.
%
%
%\begin{result}\Label{two-reg-lem} 
%Let $\R$ be a $2$-regulus in $\PG(5,q)$.  If a $3$-space $\Pi$ meets three planes of $\R$ in lines, then it meets all the planes in $\R$ in lines, 
%and $\Pi\cap\R$ is a 1-regulus of $\Pi$. 
%\end{result}
%
%The next result considers an exterior splash $\ES$, and looks at how $3$-spaces that contain a cover plane of $\ES$ meet $\ES$.  Using \cite[Lemma 10(5)]{lavr10}, as a cover plane meets every element of $\ES$, we have:
%
%\begin{result}\Label{splash-regulus}
%Let $\ES$ be an exterior splash in $\PG(5,q)$, and let $\pi$ be a cover plane of $\ES$.  
%Let $\Sigma$ be a $3$-space of $\PG(5,q)$ containing $\pi$, then $\Sigma$ meets $q+1$ planes of $\ES$ in lines, and these lines form a 1-regulus of $\Sigma$.
%% Further, $\Sigma$ meets the remaining planes of $\ES$ in 
% %in exactly one point.
%\end{result}
%}

\begin{theorem}\Label{max-int-ext}
If two exterior splashes of $\li$ in $\PG(2,q^3)$ meet in $2q+2$ points, then these points correspond to two disjoint \orsl s, one from each family.  
\end{theorem}

\begin{proof} We work in $\PG(6,q)$, so let 
 $\ES_1,\ES_2$ be two exterior splashes of the regular 2-spread $\S$, and suppose that $\ES_1,\ES_2$ meet in $2q+2$ planes. By Theorem~\ref{sublinesinBB}, we need to show that $\ES_1\!\,\!\cap\ES_2$ consists of two disjoint 2-reguli. We first show that $\ES_1\!\,\!\cap\ES_2$ contains a $2$-regulus.
  Let $\gamma$ be a plane in $\ES_1\!\,\!\cap\ES_2$ and $P$ a point in $\gamma$. 
Let $\alpha$  be a cover plane of $\ES_1$ through $P$, and let $\beta$ be a cover plane of $\ES_2$ through $P$. We consider two cases.
\begin{enumerate}
\item Suppose $\alpha\cap\beta$ is a line. 
By Lemma~\ref{L:prop-cover}, each  line of $\alpha$ (respectively $\beta$) is a ruling line of a distinct $2$-regulus of $\S$, and this 2-regulus is contained in $\ES_1$ (respectively  $\ES_2$). So the line 
 $\alpha\cap\beta$ is the ruling line of a 2-regulus contained in $\ES_1\cap\ES_2$, as required. 
 \item
 Suppose that $\alpha,\beta$ do not meet in a line, so $\alpha\cap\beta=P$. We consider two sub-cases.
 \begin{enumerate}
 \item Firstly suppose some line $\ell$ of $\alpha$ through $P$ meets at least three planes of $(\ES_1\!\,\!\cap\ES_2)\setminus\{\gamma\}$.
 By Lemma~\ref{L:prop-cover}, the line $\ell$ is a ruling line of a unique 2-regulus $\R\subset\ES_1$. Hence 
 the $3$-space $\Pi=\langle\ell,\beta\rangle$ meets at least three planes of $(\ES_1\!\,\!\cap\ES_2)\setminus\{\gamma\}$ in lines. Hence  by \cite[Lemma 7]{lavr10}, $\Pi$ meets all the planes in $\R$ in lines. Thus $\beta$ meets all the planes in $\R$, and so $\R$ also lies in $\ES_2$. So in this case, $\ES_1\!\,\!\cap\ES_2$ contains the $2$-regulus $\R$. 
\item For the second case, we 
 suppose that each line of $\alpha$ through $P$ meets at most two of the planes in $(\ES_1\!\,\!\cap\ES_2)\setminus\{\gamma\}$. As $|(\ES_1\!\,\!\cap\ES_2)\setminus\{\gamma\}  |=2q+1$,  
 $q$ of the lines in $\alpha$ through $P$ meet two
 planes of $(\ES_1\!\,\!\cap\ES_2)\setminus\{\gamma\}$, and the remaining line of $\alpha$ through $P$ meets one plane. 
A counting argument shows that  the 4-space $\langle\alpha,\beta\rangle$ contains exactly one plane $\delta$ of the regular $2$-spread $\S$,  and meets the remaining $q^3$ planes of $\S$ in a line.  That is,  $\alpha,\beta,\delta$ all lie in the 4-space $\langle\alpha,\beta\rangle$,  $\delta$ is a plane of $\S$, and $\alpha,\beta$ are cover planes of an exterior splash contained in  $\S$. Hence it follows that $\delta$ meets both $\alpha$ and $\beta$ in points. As $\delta\cap\alpha$ is a point, $\delta\in\ES_1$. Similarly, $\delta\in\ES_2$, so $\delta\in\ES_1\!\,\!\cap\ES_2$.

We now show that the point $P=\alpha\cap\beta$ lies in $\delta$. Suppose not, that is, $P\notin\delta$.
Now one line of $\alpha$ through $P$ contains the point $\delta\cap\alpha$, and (by the assumption above) $q$ lines of $\alpha$ through $P$ meet two planes of $(\ES_1\!\,\!\cap\ES_2)\setminus\{\gamma\}$. So let 
$\ell$ be a line of $\alpha$ through $P$, not through $\delta\cap\alpha$, that meets two planes  $\gamma_1,\gamma_2$ of $(\ES_1\!\,\!\cap\ES_2)\setminus\{\gamma\}$ in the points $G_1,G_2$ respectively.  Hence the $3$-space $\Pi=\langle\ell,\beta\rangle$ 
 meets $\gamma_1$ and $\gamma_2$ in lines. Further, both 
 $\delta$ and $\Pi$ lie in the 4-space $\langle\alpha,\beta\rangle$, so $\Pi$ meets $\delta$ in a line. 
 That is, $\Pi$ meets three planes $\gamma_1,\gamma_2,\delta$ of $(\ES_1\!\,\!\cap\ES_2)\setminus\{\gamma\}$ in  lines. These three lines lie in a unique 1-regulus $\mathscr R'$ of $\Pi$.
So $\Pi\cap\alpha$ contains the three non-collinear points $G_1$, $G_2$, and $\delta\cap \alpha$.  Hence $\alpha\subset\Pi$, contradicting $\alpha\cap \beta$ being exactly the point $P$.
Hence the point $P$ lies in $\delta$, and so $\delta=\gamma$. 
So $\Pi$ meets the three planes $\gamma_1,\gamma_2,\gamma$ of $\ES_1\cap\ES_2$ in lines, and these three lines lie in a unique 1-regulus 
 $\mathscr R'$ of $\Pi$. So $\mathscr R'$ 
   lies in the $2$-regulus $\R$ defined by $\gamma_1,\gamma_2,\gamma$. As $\ell$ meets each line of 
   $\mathscr R'$, 
   $\ell$ is a ruling line of $\R$ and so $\R\subset\ES_1$. Moreover,  as $\mathscr R'=\R\cap\Pi$,  $\beta$ meets each plane of $\R$, and so $\R\subset\ES_2$. That is, $\R$ is a $2$-regulus in $\ES_1\!\,\!\cap\ES_2$.
\end{enumerate}
\end{enumerate}
So we have shown that $\ES_1\!\,\!\cap\ES_2$ contains a $2$-regulus $\R$.  Let $\ell$ be a ruling  line of the $2$-regulus $\R$. Then by Lemma~\ref{L:prop-cover}, there is a unique cover plane $\alpha_1$ of $\ES_1$ through $\ell$, and a unique cover plane $\alpha_2$ of $\ES_2$ through $\ell$. 
So $\alpha_1$ meets $\ES_1\!\,\!\cap\ES_2\setminus\R$ in $q+1$ points of $\alpha_1\backslash\ell$, and similarly for $\alpha_2$.  
So the 3-space $\langle\alpha_1,\alpha_2\rangle$ meets $\ES_1\!\,\!\cap\ES_2\setminus\R$ in $q+1$ lines. By
\cite[Lemma 10]{lavr10}, these $q+1$ lines form a 1-regulus we denote $\mathscr R_1$. 
Let $m$ be a transversal line of $\mathscr R_1$, Then $m$ meets $q+1$ planes of $\ES_1\!\,\!\cap\ES_2\setminus\R$. Now $m$ is the ruling line of a unique 2-regulus $\mathcal R'$  which is contained in both $\ES_1$ and $\ES_2$, hence $\mathcal R'=\ES_1\!\,\!\cap\ES_2\setminus\R$.
So in $\PG(6,q)$, $\ES_1\!\,\!\cap\ES_2$ contains two disjoint 2-reguli $\R,\R'$. Hence,  by Theorem~\ref{sublinesinBB},  in $\PG(2,q^3)$, $\ES_1\!\,\!\cap\ES_2$ contains two disjoint \orsls. As the \orsls\ are  disjoint, they belong to different families. 
\end{proof}

\begin{corollary} Let $\ES_1$, $\ES_2$ be two exterior splashes of $\PG(6,q)$ that meet in $2q+2$ planes. Then each cover plane  in the cover $\EX$ of $\ES_i$ meets $\ES_1\!\,\!\cap\ES_2$ in a set of $2q+2$ points that form a line and an $\EX$-special conic. 
\end{corollary}

\begin{proof}
We continue from the end of the proof of Theorem~\ref{max-int-ext}, using the same notation as in the last paragraph of the proof.  In particular,  the 3-space denoted $\langle\alpha_1,\alpha_2\rangle$ meets $\ES_1\cap\ES_2$ in a 1-regulus $\mathscr R_1$; and  $\mathscr R_1$ meets $\alpha_1\setminus\ell$ in $q+1$ points. 
 As a 1-regulus meets a  plane in either two lines, or in an irreducible conic, $\mathscr R_1$  meets $\alpha_1$ in an irreducible conic $\C$ exterior to $\ell$. 
The points of $\C$ lie in $q+1$ distinct planes of $\ES_1\cap\ES_2$, namely the planes of the  2-regulus denoted  $\mathcal R'$  above.
 Hence by \cite[Theorems  7.2, 7.3]{BJ-ext2} $\C$ is an $\EX$-special conic, where $\EX$ is the cover of $\ES$ that contains $\alpha_1$. So $\ES_1\!\,\!\cap\ES_2$ meets $\alpha_1$ in the line $\ell$ and an $\EX$-special conic. 
A similar result holds for the plane $\alpha_2$. Moreover, this situation generalises to any cover plane of $\ES_i$.
\end{proof}

\section{Conclusion}

This article completes our investigation of exterior \orsps\ of $\PG(2,q^3)$ begun in \cite{BJ-ext1,BJ-ext2}. We note that the proofs in this article are a mix of geometric arguments and coordinate based arguments. It would be nice to exploit the underlying geometry of the structure to prove more results using a geometric argument.

%
%Let $\pi$ be an exterior \orsp\ with exterior splash $\ES$, conic cover $\EC$ and tangent cover $\ET$.
%
%\begin{tabular}{l|l}
%$\PG(6,q)$&$\PG(2,q^3)$\\ \hline\hline
%$\EC$-special conics (in a plane of $\EC$)& $\pi$-\psline s of $\ES$   \\
%$\ET$-special conics &$\pi$-\dcsline s of $\ES$      \\
%$\ES$-special conics &   one point on $\li$ \\
%$\S$-special conics &    one point on $\li$ \\ \hline
%$\S$-special twisted cubics& exterior \orsls\ of $\PG(2,q^3)$   \\
%$\ES$-special twisted cubics&  exterior \orsls\ of $\pi$  \\
%$\EC$-special twisted cubics&  $(\pi,\li)$-special conics of $\pi$  \\
%&   OR $\leq 3$ points in $\pi$ and a $(\pi',\li)$-special conics of some other $\pi'$ \\
%
%$\ET$-special twisted cubics&   at most a triangle in $\pi$ \\
%$\ET$-special twisted cubics&   $(\pi',\li)$-special conics of some other $\pi'$ \\
%\end{tabular}

A recurrent theme of the results
on exterior \orsps\ is that the conics and twisted cubics of greatest interest are ``special''. To summarise, we have for a fixed exterior splash $\ES$ with conic cover $\EC$ and tangent cover $\ET$:
\begin{itemize}
\item $\EC$-special conics in a plane of $\EC$ (and $\ET$-special conics in a plane of $\ET$) correspond to \orsls\ of $\ES$ (in different families).
\item $\ES$-special twisted cubics correspond to lines of an associated \orsp.
\item $\EC$-special twisted cubics (and $\ET$-special twisted cubics) correspond to special conics in some associated \orsp.
\end{itemize}
Each result has been proved individually, however, it would be interesting to have a unifying theory to explain 
these relationships. 
We note that Theorem~\ref{thm:quad-trans}
 (where we show that the nine quadrics  defining $[\pi]$ in $\PG(6,q)$ all  contain the transversals $\gs$ and $\gc$) goes someway towards understanding this.

\end{document}